\documentclass[12pt,twoside]{article}

\usepackage{vmargin}
\usepackage{fancyhdr}
\usepackage{ifthen}
\usepackage{graphicx} 
\usepackage{indentfirst}
\usepackage[english]{babel}

\usepackage{fouriernc}
\DeclareMathAlphabet{\mathcal}{OMS}{cmsy}{m}{n} %otherwise mathcal looks like mathscr

\usepackage{amsmath}   
\usepackage{amssymb}   
\usepackage{amsfonts}
\usepackage{amsthm}
\usepackage{mathrsfs}
\usepackage{pst-node} 
\usepackage{pstricks}

\usepackage{tweaklist} 
\renewcommand{\itemhook}{\setlength{\topsep}{1pt}%
\setlength{\itemsep}{-1pt}}

\renewcommand{\enumhook}{\setlength{\topsep}{1pt}%
\setlength{\itemsep}{-1pt}}

\usepackage[colorlinks=true,linkcolor=magenta,citecolor=blue]{hyperref}

\setpapersize{A4}
\setmarginsrb{33mm}{20mm}{33mm}{25mm}{10mm}{6mm}{0mm}{10mm}
\fancypagestyle{plain}{%
\fancyhf{} % clear all header and footer fields
\lfoot[\thepage]{}
\cfoot[]{}
\rfoot[]{\thepage}
}

\newcommand{\sk}{\smallskip}
\newcommand{\mk}{\medskip}
\newcommand{\bk}{\bigskip}
\renewcommand{\emptyset}{\ensuremath{\varnothing}}     

\newcommand{\G}{\mathbf{G}}
\newcommand{\B}{\mathbf{B}}
\newcommand{\U}{\mathbf{U}}

\newcommand{\T}{\mathbf{T}}
\renewcommand{\P}{\mathbf{P}}
\renewcommand{\L}{\mathbf{L}}

\newcommand{\intn}[2]{\ensuremath{[\![ \, #1 \,;\, #2 \,]\!]}}

\newlength{\leftlength}
\newlength{\rightlength}
\newlength{\calculskip}
\newcommand{\calculvskip}[1]{%
  \ifthenelse{#1 = 0}{\setlength{\calculskip}{0pt}}{}%
  \ifthenelse{#1 = 1}{\setlength{\calculskip}{\smallskipamount}}{}%
  \ifthenelse{#1 = 2}{\setlength{\calculskip}{\medskipamount}}{}%
  \ifthenelse{#1 = 3}{\setlength{\calculskip}{\bigskipamount}}{}%
  \ifthenelse{#1 = 4}{\setlength{\calculskip}{1cm}}{}%
  \vskip\calculskip
}

\newcommand{\leftcentersright}[4][2]{%
        \settowidth{\leftlength}{#2}%
        \settowidth{\rightlength}{#4}%
        \calculvskip{#1}
        \noindent#2\hskip-\leftlength%
        \hfill#3\hfill
        \mbox{}\hskip-\rightlength#4%
        \vskip\calculskip%
        }

\newcommand{\centers}[2][2]{\leftcentersright[#1]{}{#2}{}}
\newcommand{\leftcenters}[3][2]{\leftcentersright[#1]{#2}{#3}{}}

\makeatletter
\def\svhline{%
  \noalign{\ifnum0=`}\fi\hrule \@height2\arrayrulewidth \futurelet
   \reserved@a\@xhline}
   
\def\hlinewd#1{%
\noalign{\ifnum0=`}\fi\hrule \@height #1 %
\futurelet\reserved@a\@xhline}
 \makeatother

\numberwithin{equation}{section}

\newtheorem{prop}[equation]{Proposition}  
\newtheorem{de}[equation]{Definition}
\newtheorem{thm}[equation]{Theorem}
\newtheorem{lem}[equation]{Lemma}  
\newtheorem{cor}[equation]{Corollary}

\newtheorem{conj}[equation]{Conjecture}
\newtheorem*{conjHLM}{Conjecture HLM}
\newtheorem*{conjHLM2}{Conjecture HLM+}
\newtheorem*{thm2}{Theorem B}
\newtheorem*{thm3}{Theorem A}

\theoremstyle{definition}
\newtheorem{exemple}[equation]{Example}  
  
\newtheorem{rmk}[equation]{Remark}  
\newtheorem{fait}[equation]{Fact}  
\newtheorem{cons}[equation]{Consequences}  
  
\newtheorem{notation}[equation]{Notation}

\renewcommand\a{\mathscr{A}}
\renewcommand\b{\mathscr{B}}
\newcommand{\X}{\mathrm{X}}
\newcommand{\Y}{\mathrm{Y}}
\newcommand{\Hc}{\mathrm{H}_c}
\newcommand{\Irr}{\mathrm{Irr}}
\newcommand{\Rgc}{\mathrm{R}\Gamma_c}
\newcommand\dec{\mathrm{dec}}
\newcommand\mathcalm[1]{\fontfamily{cmss}\selectfont \mathcal{#1}}
\newcommand\hg{\mathrm{h}_\Gamma}
\newcommand{\ol}{\mathop{\otimes}\limits^\mathrm{L}}

\usepackage{pstricks-add}

\begin{document}

\title{Coxeter orbits and Brauer trees}
\author{Olivier Dudas\footnote{Mathematical Institute, Oxford.}
\footnote{The author is supported by the EPSRC, Project No EP/H026568/1, by Magdalen College, Oxford and partly by the ANR, Project No JC07-192339.}}

\maketitle

\begin{abstract} We study the cohomology with modular coefficients of Deligne-Lusztig varieties associated to Coxeter elements. Under some torsion-free assumption on the cohomology we derive several results on the principal $\ell$-block of a finite reductive group $\G(\mathbb{F}_q)$ when the order of $q$ modulo $\ell$ is assumed to be the Coxeter number. These results include the determination of the planar embedded Brauer tree of the block (as conjectured by Hiss, L\"ubeck and Malle in \cite{HLM}) and the derived equivalence predicted by  the geometric version of Brou\'e's conjecture \cite{BMa2}.
\end{abstract}

\section*{Introduction}

Let $\G$ be a quasi-simple algebraic group defined over an algebraic closure of a finite field of characteristic $p$. Let $F$ be the Frobenius endomorphism of $\G$ associated to a rational $\mathbb{F}_q$-structure. The finite group $G = \G^F$ of fixed points under $F$ is called a finite reductive group.

\sk

The ordinary representation theory of $G$ is now widely understood: geometric methods have been developed by Deligne and Lusztig \cite{DeLu} and then extensively studied by Lusztig, leading to a complete classification of the irreducible characters of $G$ \cite{Lu5}. One of the key step in Lusztig's fundamental work has been to determine explicitly the $\ell$-adic cohomology of  Deligne-Lusztig varieties associated to Coxeter elements \cite{Lu}. This paper is an attempt to extend  this  work to the modular setting. To be more precise, we will be interested in the complex $\Rgc(\Y(\dot c),\Lambda)$ representing the cohomology of the variety $\Y(\dot c)$ with coefficients in a finite extension $\Lambda$ of $\mathbb{Z}_\ell$, and more specifically in the action of $G$ and $F$ on this complex. The representation theory of $\Lambda G$ is highly dependent on the prime number $\ell$. In this particular geometric situation $-$ the Coxeter case $-$ the corresponding primes $\ell$ are those which divide the cyclotomic polynomial $\Phi_h(q)$ where $h$ is the Coxeter number. For such prime numbers, the principal part $b\Rgc(\Y(\dot c), \Lambda)$ of the cohomology complex should encode many aspects of the representation theory of the principal $\ell$-block $b$ of $G$, much of which remains conjectural.

\sk

In the Coxeter case, the principal $\ell$-block of $G$ has a cyclic defect group. From the work of Brauer we know that the category of modules over such a block has a combinatorial description, given by the \emph{Brauer tree}. The shape of this tree is related to the decomposition matrix of the block whereas its planar embedding determines the module category up to Morita equivalence.  By a case-by-case analysis, Hiss, L\"ubeck and Malle have underlined in \cite{HLM} a deep connection between the Brauer tree of the principal $\ell$-block and the $\ell$-adic cohomology groups of the Deligne-Lusztig variety $\X(c)$.
They have conjectured that this connection remains valid even in the cases where the shape of the tree or its planar embedding is not explicitly known (see Conjectures \hyperref[2conj1]{(HLM)} and  \hyperref[2conj2]{(HLM+)}). Furthermore, they have suggested that  the cohomology with modular coefficients should give enough information to prove the conjecture in its full generality. This is the geometric approach that we will follow throughout this paper. It relies on the fact that the complex $b\Rgc(\Y(\dot c),\Lambda)$ is  perfect and thus provides many projective modules. 

\sk

In order to determine the shape of the Brauer tree, we must find every indecomposable projective $\Lambda G$-module lying in the principal block and compute their characters. Such projective modules admit no direct algebraic construction via Harish-Chandra induction, because they might have a cuspidal head. Drawing inspiration from Lusztig's work \cite{Lu}, we show that they can be obtained by taking suitable eigenspaces of $F$ on $b\Rgc(\Y(\dot c),\Lambda)$. However, there is a price to pay since we have to make the following assumption:

\centers{\begin{tabular}{cp{13cm}} \hskip-1mm $\mathbf{(W)}$ &  For all minimal eigenvalues $\lambda$ of  $F$,  the generalized $(\lambda)$-eigenspace of $F$ on $b\Hc^\bullet(\Y(\dot c),\Lambda)$ is torsion-free. \end{tabular}}

\noindent We call here "minimal" the eigenvalues of $F$ on the cohomology group in middle degree. Under this precise assumption, in Section 3 we give  a general proof of Conjecture \hyperref[2conj1]{(HLM)}
of Hiss-L\"ubeck-Malle. 

\sk

There is strong evidence that the previous assumption should be valid for any eigenvalue of $F$ \cite{Du2} which means that the following should hold:

\centers{\begin{tabular}{cp{13cm}} \hskip-1mm $\mathbf{(S)}$ &  The $\Lambda$-modules $b\Hc^i(\Y(\dot c),\Lambda)$ are torsion-free. \end{tabular}}

\noindent This technical assumption will be discussed in a subsequent paper. Using generalized Gelfand-Graev representations we can actually prove that it holds in a majority of cases, which justifies our approach:

\begin{thm3}[\cite{Du3}] Assume that $\G$ has no factor of type $E_7$ or $E_8$ and that $p$ is a good prime number. Then assumption $(\mathrm{S})$ holds.
\end{thm3}

\noindent Note that the torsion-free property is really specific to the case of Deligne-Lusztig varieties associated to Coxeter elements. The cohomology of varieties associated to other elements should have a non-trivial torsion component in general (see \cite[Remark 5.37]{Du2}).

\pagebreak

 In this paper, we shall instead focus on all the consequences we can draw from such an assumption. Our main result is the complete determination of the complex $b\Rgc(\Y(\dot c),\Lambda)$ in terms of the projective modules lying in the block.  Surprisingly, our representative  turns out to be exactly the complex attached to a Brauer tree in \cite{Ri}:

\begin{thm2} Under the assumption $\mathrm{(S)}$, the complex $b\Rgc(\Y(\dot c),\Lambda)$ is homotopy equivalent to the Rickard complex associated to the Brauer tree of the principal $\ell$-block of $G$, shifted by the length of $c$.
\end{thm2}

\noindent From that observation we deduce several important results:

\begin{itemize} 

\item[$\bullet$] the cohomology complex induces the derived equivalence predicted by the geometric version of Brou\'e's conjecture \cite{BMa2} (see Theorem \ref{4thm3});

\item[$\bullet$] this equivalence is perverse in the sense of \cite{CR1} (see Theorem \ref{4thm4});

\item[$\bullet$] the planar embedding of the Brauer tree can be read out from the eigenvalues of $F$ (see Theorem \ref{4thm2}). 

\end{itemize}

\noindent Together with Theorem A,  this gives new results for the geometric version of Brou\'e's conjecture.  This extends significantly the previous work of Puig \cite{Pu2} (for $\ell \, | \, q-1$), Rouquier \cite{Rou} (for $\ell \, | \, \phi_h(q)$ and $r=1$) and Bonnaf\'e-Rouquier \cite{BR2} (for $\ell \, |\, \phi_h(q)$ and $(\G,F)$ of type A$_n$). We also obtain new planar embedded Brauer tree for groups of type ${}^2$G$_2$ and F$_4$ with $p \neq 2,3$ (compare with \cite{Hiss} and \cite{HL}).

\sk

 This paper is divided into four parts: in the first section, we introduce basic notation and recall standard constructions in the modular representation theory of finite reductive groups, formulated in the language of derived categories. In Section 2 we present what we call the \emph{Coxeter case} and collect different results (both geometric and group-theoretic) that have been obtained so far for principal $\ell$-blocks and their Brauer trees in this particular case. Section 3 is devoted to the proof of the conjecture \hyperref[2conj1]{(HLM)} of Hiss, L\"ubeck and Malle. We show that under the assumption $\mathrm{(W)}$, the Brauer tree can be deduced from the $\ell$-adic cohomology of $\X(c)$. Finally, we use assumption $\mathrm{(S)}$ in Section 4 to determine an explicit representative of the cohomology complex. As a byproduct, we obtain a proof of the geometric version of Brou\'e's conjecture (as always in the Coxeter case) as well as and the planar embedding of the Brauer tree.

\section{Preliminaries}

\subsection{Some homological algebra\label{1se1}}

We start by recalling some standard notions of homological algebra that we will use throughout this paper.

\mk

\noindent \thesubsection.1. \textbf{Module categories and usual functors.} If $\mathscr{A}$ is an abelian category, we will denote by $C(\mathscr{A})$ the category of cochain complexes, by $K(\mathscr{A})$ the corresponding homotopy category and by $D(\mathscr{A})$ the derived category. We shall use the superscript notation $-$, $+$ and $b$ to denote the full subcategories of bounded above, bounded below or bounded complexes. We will always consider the case where $\mathscr{A} = A$-$\mathrm{Mod}$ is the module category over any ring $A$ or the full subcategory $A$-$\mathrm{mod}$ of finitely generated modules. This is actually not a strong restriction, since any small category can be embedded into some module category \cite{Mit}. Since the categories $A$-$\mathrm{Mod}$ and $A$-$\mathrm{mod}$ have enough projective objects, one can define the usual derived bifunctors $\mathrm{RHom}_{A}^\bullet$ and ${\ol}_A$. 

\sk

Let $H$ be a finite group and $\ell$ be a prime number.  We fix an $\ell$-modular system $(K, \Lambda, k)$ consisting of a finite extension $K$ of the field of $\ell$-adic numbers $\mathbb{Q}_\ell$, the integral closure $\Lambda$ of the ring of $\ell$-adic integers in $K$ and the residue field $k$ of the local ring $\Lambda$. We assume moreover that the field $K$ is big enough for $H$, so that it contains the $e$-th roots of unity,  where $e$ is the exponent of $H$. In that case, the algebra $KH$ is split semi-simple. 

\sk

From now on, we shall focus on the case where $A = \mathcal{O} H$, with $\mathcal{O}$ being any ring among $(K,\Lambda,k)$. By studying the modular representation theory of $H$ we mean studying the module categories $\mathcal{O} H$-$\mathrm{mod}$ for various $\mathcal{O}$, and also the different connections between them. In this paper, most of the representations will arise in the cohomology of some complexes and we need to know how to pass from one coefficient ring to another. The scalar extension and $\ell$-reduction have a derived counterpart: if $C$ is any bounded complex of $\Lambda H$-modules we can form $KC = C \otimes_\Lambda K$ and $\overline{C} = kC = C {\ol}_\Lambda k$. Since $K$ is a flat $\Lambda$-module, the cohomology of the complex $KC$ is exactly  the scalar extension of the cohomology of $C$. However this does not apply to $\ell$-reduction, but the obstruction can be related to the torsion:

\begin{thm}[Universal coefficient theorem]\label{1thm1}Let $C$ be a bounded complex of $\Lambda H$-modules. Assume that the terms of $C$ are free over $\Lambda$. Then for all $n \geq 1$ and $i \in \mathbb{Z}$, there exists a short exact sequence of $\Lambda H$-modules

\centers[1]{$0 \longrightarrow \mathrm{H}^i(C)\otimes_\Lambda \Lambda / \ell^n \Lambda  \longrightarrow \mathrm{H}^i\big(C \, {\ol}_\Lambda\, \Lambda / \ell^n \Lambda \big) \longrightarrow \mathrm{Tor}_1^\Lambda(\mathrm{H}^{i+1}(C), \Lambda/\ell^n \Lambda) \longrightarrow 0.$}

\end{thm}

In particular, whenever there is no torsion in both $C$ and $\mathrm{H}^\bullet(C)$ then the cohomology of $\overline{C}$ is exactly the $\ell$-reduction of the cohomology of $C$.

% By studying the modular representation theory of $H$ we mean studying the categories $\mathcal{O} H$-$\mathrm{mod}$ for various $\mathcal{O}$ and how there are related.

\mk

\noindent \thesubsection.2. \textbf{Perfect complexes.} In the derived category, any acyclic complex is isomorphic to zero. This can be generalized  to the homotopy category if we restrict ourselves to a specific class of complexes. Recall that a complex $C$ of $\mathcal{O} H$-modules is said to be \emph{perfect} if it is quasi-isomorphic to a bounded complex of finitely generated projective $\mathcal{O} H$-modules. The value of a derived bifunctor on any perfect complex is obtained from the original functor: more precisely, if $C$ is a perfect complex and $C'$ is any complex in $C(\mathcal{O} H$-$\mathrm{Mod})$ then there exists isomorphisms in $D(\mathbb{Z}$-$\mathrm{Mod})$:

\centers{$ C \, {\ol}_{\mathcal{O}H} \, C' \, \simeq \, C \otimes_{\mathcal{O}H} C' \qquad \text{and} \qquad \mathrm{RHom}_{\mathcal{O}H}^\bullet(C,C') \, \simeq \, \mathrm{Hom}_{\mathcal{O}H}^\bullet(C,C').$}
 
\noindent In particular the natural functor $K^b(\mathcal{O}H$-$\mathrm{proj}) \longrightarrow D(\mathcal{O}H$-$\mathrm{Mod})$ induces an equivalence between the homotopy category of bounded complex of finitely generated projective modules and the full subcategory of $D(\mathcal{O}H$-$\mathrm{Mod})$ of perfect complexes.

\sk

As in the derived case, one can obtain precise concentration properties for complexes in $K^b(\mathcal{O}H$-$\mathrm{proj})$:

\begin{lem}\label{1lem1}Let $C \in K^b(\mathcal{O}H$-$\mathrm{proj})$. Assume that the cohomology of $C$ vanishes outside the degrees $m, \ldots,M$. Then $C$ is homotopy equivalent to a complex

\centers{$ \cdots \longrightarrow 0 \longrightarrow P_{m-1} \longrightarrow  P_m \longrightarrow \cdots \longrightarrow P_M \longrightarrow 0 \longrightarrow \cdots $}

\noindent whose terms are finitely generated projective modules, concentrated in degrees $m-1,m,\ldots,M$. Moreover, $P_{m-1}$ can be chosen to be zero in the following cases:

\begin{itemize}

\item[$\mathrm{(a)}$] $\mathcal{O}$ is a field;

\item[$\mathrm{(b)}$] $\mathcal{O}=\Lambda$ and the group $\mathrm{H}^m(C)$ is torsion-free.

\end{itemize}
\end{lem}

\begin{proof} Let us write $C$ as  $\cdots \longrightarrow P_r \mathop{\longrightarrow}\limits^{d_r} P_{r+1} \longrightarrow 0$ with $r \geq M$. By assumption, the cohomology of $C$ vanishes in degree $r+1$. The boundary map $d_r$ is surjective and splits since $P_{r+1}$ is a projective module. Therefore the complex $0 \longrightarrow P_{r+1} \longrightarrow P_{r+1} \longrightarrow 0$ is a direct summand of $C$, and being homotopy equivalent to zero it can be removed. 

\sk

If $\mathcal{O}$ is one of the fields $K$ or $k$, then every projective $\mathcal{O} H$-module is injective, and we can again remove successively the terms $P_i$ for $i < m$. The case where $\mathcal{O} = \Lambda$ requires more attention. The complex $C$ can be written as $0 \longrightarrow P_r \mathop{\longrightarrow}\limits^{d_r} P_{r+1} \longrightarrow  \cdots $ with $r < m$ and the map $d_r$ being injective. We claim that there exists a retraction of $d_r$ if  $\mathrm{Coker}\, d_r$ is torsion-free. Indeed, $P_r$ is $(H,1)$-injective by \cite[Theorem 19.12]{CuRe} so that the short exact sequence $0 \longrightarrow P_r \longrightarrow P_{r+1} \longrightarrow \mathrm{Coker}\, d_r \longrightarrow 0$ splits over $\Lambda H$ whenever it splits over $\Lambda$. Now, if $r < m-1$ then  by assumption $\mathrm{Im} \, d_r = \mathrm{Ker} \, d_{r+1}$ so that $\mathrm{Coker} \, d_r \simeq \mathrm{Im} \, d_{r+1}$ is torsion-free (it is a submodule of a free module). If $m=r-1$, we can observe that the quotient  $\mathrm{Coker} \, d_{m-1} / \mathrm{H}^m(C) \simeq \mathrm{Im}\, d_m$ is torsion-free, and hence $\mathrm{Coker} \, d_{m-1} $ is torsion-free whenever $\mathrm{H}^m(C)$ is. 
\end{proof}

\subsection{Cohomology of a quasi-projective variety\label{1se2}}

In the ordinary Deligne-Lusztig theory, representations of finite reductive groups arise from the cohomology of some varieties acted on by the group. In the modular setting, we shall be interested not only in the cohomology of theses varieties but also in  complexes representing the cohomology in the derived category. This gives much more information, some of which has already been collected by Brou\'e \cite{Bro1}, \cite{BMa2}, and more recently by Bonnaf\'e and Rouquier \cite{BR1}.

\mk

\noindent \thesubsection.1. \textbf{Rouquier's construction.} Let $\X$ be a quasi-projective algebraic variety defined over $\overline{\mathbb{F}}_p$. We assume that $\X$ is  acted on by a finite group $H$ and by a monoid of endomorphisms $\Upsilon$ normalizing $H$. We shall always assume that the prime number $p$ (the defining characteristic of all the varieties involved) is different from $\ell$ (associated to the modular system). 

\sk

Let  $\mathcal{O}$ be any ring among $(K,\Lambda,k)$. Rouquier has constructed in \cite{Rou} a bounded complex $\Rgc(\X,\mathcal{O})$ of $\mathcal{O} H \rtimes \Upsilon$-modules representing the $\ell$-adic cohomology with compact support of $\X$. In other words, we have $\mathrm{H}^\bullet\big(\Rgc(\X,\mathcal{O})\big) \simeq \Hc^\bullet(\X,\mathcal{O})$ as $\mathcal{O} H \rtimes \Upsilon$-modules. This construction is particularly adapted to the modular setting, since the cohomology complexes behave well with respect to scalar extension and $\ell$-reduction. We have indeed in $D^b(\mathcal{O}H\rtimes \Upsilon$-$\mathrm{mod})$:

\centers{$  \Rgc(\X,\Lambda) \, {\ol}_{\Lambda } \, \mathcal{O} \, \simeq \, \Rgc(\X, \mathcal{O}).$}

\noindent In particular, the universal coefficient theorem (Theorem \ref{1thm1}) will hold for $\ell$-adic cohomology with compact support.

\sk

Let us forget about the action of $\Upsilon$ for the moment. By construction, the terms of  $\Rgc(\X,\Lambda)$ are far from being finitely generated. Nevertheless, we can, up to homotopy equivalence, find a representative with good finiteness properties \cite[Section 2.5.1]{Rou}:

\begin{thm}[Rouquier]\label{1thm2}Let $\X$ be a quasi-projective variety acted on by a finite group $H$. Denote by $\mathcal{S} = \{\mathrm{Stab}_H(x)\, | \, x\in \X\}$ the set of stabilizers of points of $\X$. Then $\Rgc(\X,\mathcal{O})$ is homotopy equivalent to a complex $\mathscr{C}$  whose terms are direct summands of finite sums of permutation modules $\mathcal{O}[H/S]$ for various $S \in \mathcal{S}$.
\end{thm}

\begin{rmk}\label{1rmk1}Note that $\mathscr{C}$ is not a complex of $\mathcal{O}H \rtimes \Upsilon$-modules for it can only inherit an action of $\Upsilon$ up to homotopy. 
\end{rmk}

\begin{cor}\label{1cor1}Assume that the order of the stabilizer of any point is invertible in $\mathcal{O}$. Then $\Rgc(\X,\mathcal{O})$ is a perfect complex of $\mathcal{O} H$-modules.
\end{cor}

Recall that the $\ell$-adic cohomology with compact support of any irreducible affine variety of dimension $d$ is concentrated in degrees $d,\ldots,2d$, and this for any coefficient ring among $(K,\Lambda,k)$. Using Lemma \ref{1lem1} and Theorem \ref{1thm1} we deduce the following:

\begin{cor}\label{1cor2}Let $\X$ be an irreducible affine variety of dimension $d$. Assume that the order of the stabilizer of any point is prime to $\ell$. Then $\Rgc(\X,\Lambda)$ can be represented by a complex of finitely generated projective $\Lambda H$-modules concentrated in degrees $d,\ldots,2d$. In particular, $\Hc^d(\X,\Lambda)$ is torsion-free.
\end{cor}

Note that the result holds for any disjoint union of irreducible affine varieties with the same dimension, and therefore for any Deligne-Lusztig variety that has been proven to be affine.

\mk

\noindent \thesubsection.2. \textbf{Generalized eigenspaces of the Frobenius.} We now study the case where  $\Upsilon = \langle F \rangle_{\mathrm{mon}}$ is generated by the Frobenius endomorphism attached to some rational $\mathbb{F}_q$-structure on $\X$. We would like to factor out the complex $\Rgc(\X,\mathcal{O})$ with respect to the eigenvalues of $F$.  For this purpose, we shall first review some basics about endomorphisms of finitely generated $\Lambda$-modules.

\sk

Let $M$ be a $\Lambda[T]$-module.  Denote by $f$ the endomorphism of $M$ induced by $T$. Assume that $M$ is finitely generated over $\Lambda$. Then there exists a monic polynomial $P \in \Lambda[T]$ such that $P(f) = 0$, and we are reduced to studying the action of the finite dimensional algebra $\Lambda[T]/(P)$ on $M$. We may assume without loss of generality that $P$ splits overs $\Lambda$. In that case, the factorization of $P= (T-\lambda_1)^{\alpha_1} \cdots (T-\lambda_n)^{\alpha_n}$ yields a decomposition of the module $KM$ with respect to the generalized eigenspaces of $f$:

\centers{$ KM \, = \, \mathrm{Ker}\, (f-\lambda_1)^{\alpha_1} \oplus \cdots \oplus \mathrm{Ker}\, (f-\lambda_n)^{\alpha_n}.$}

\noindent In order to obtain an modular analog of this decomposition we have to group together the eigenvalues according to their $\ell$-reduction (this becomes clear if we consider the module $\overline{M} = M \otimes_\Lambda k$). More precisely, if we define the polynomials 

 \centers{$ P_{\bar{\lambda}}(T) \, = \, \displaystyle \prod_{\bar{\lambda}_i = \bar{\lambda}} (T-\lambda_i)^{\alpha_i}$}

\noindent then the block decomposition of the algebra $\Lambda[T]/(P)$ is given by

\centers{$\Lambda[T]/(P) \simeq \displaystyle \prod_{\bar{\lambda} \in k} \Lambda[T] / (P_{\bar{\lambda}}). $}

\noindent For $\lambda \in K$, we define the \emph{generalized $(\lambda)$-eigenspace} of $f$ in $M$ to be $M_{(\lambda)} = e_{\lambda} M$ where $e_\lambda$ is the idempotent associated to the term $\Lambda[T]/(P_{\bar\lambda})$. By construction, $M$ decomposes into

\centers{$ M \, = \, \displaystyle \bigoplus_{\bar{\lambda} \in k} M_{(\lambda)}. $}

\begin{rmk}\label{1rmk2}This definition does not depend on $P$ since the module $e_\lambda M$ depends only on the image of $e_\lambda$ in the algebra $\Lambda[T]/\mathrm{Ann}(f)$. 
\end{rmk}

Equivalently, one could have defined the module $M_{(\lambda)}$ to be the kernel of the endomorphism $P_{\bar \lambda}(f)$. One can easily deduce the following result using this description:

\begin{lem}\label{1lem2}Let $f$ and $g$ be two endomorphisms of $M$. If $\overline{f} - \overline{g}$ is a nilpotent endomorphism of $\overline{M}$, then the generalized $(\lambda)$-eigenspaces of $f$ and $g$ on $M$ coincide. 
\end{lem}

The definition of $(\lambda)$-eigenpaces can be extended to the case where $\mathcal{O}$ is one of the field $K$ or $k$ by setting $(M \otimes_\Lambda \mathcal{O})_{(\lambda)} \, = \, M_{(\lambda)} \otimes_\Lambda \mathcal{O}$. The following proposition describes the relation between these modules and the usual generalized eigenspaces:

\begin{prop}\label{1prop1}Let $\lambda \in \Lambda$.

 \begin{itemize} 

\item[$\mathrm{(i)}$] The $K[T]$-module $(KM)_{(\lambda)} :=  M_{(\lambda)} \otimes_\Lambda K$ is isomorphic to the direct sum  of all the generalized $\mu$-eigenspaces where  $\mu$ runs over the set of eigenvalues congruent to $\lambda$ modulo $\ell$.

\item[$\mathrm{(ii)}$] The $k[T]$-module $\overline{M}_{(\lambda)} :=  M_{(\lambda)} \otimes_\Lambda k$ is isomorphic to the generalized $\bar{\lambda}$-eigen\-space corresponding to $\bar{\lambda} \in k$.

\end{itemize}

\end{prop}
 
More generally, if $\mathcal{O}$ is any ring among $(K,\Lambda,k)$, one can define an endofunctor $C \longmapsto C_{(\lambda)}$ of the category of bounded complexes of $\Lambda[T]$-modules finitely generated over $\Lambda$. It is an exact functor, and as such it satisfies

\centers{$ \mathrm{H}^\bullet(C_{(\lambda)}) \, \simeq \,  \mathrm{H}^\bullet(C)_{(\lambda)}$.}

Now, in order to apply this construction to the cohomology complex $\Rgc(\X,\mathcal{O})$ we need some finiteness conditions. These are given by Theorem \ref{1thm2}: there exists a bounded complex $\mathscr{C}$ of finitely generated $\mathcal{O} H$-modules, together with $H$-equivariant morphisms $f : C \longrightarrow \mathscr{C}$ and $g : \mathscr{C} \longrightarrow C$ which are mutually inverse in the category $K^b(\mathcal{O}H$-$\mathrm{Mod})$. Assume that the Frobenius $F$ commutes with the action of $H$ so that we can define a $H$-equivariant morphism on $\mathscr{C}$ by setting $\mathscr{F} = f \circ F \circ g$. The definition of $\mathscr{F}$ depends on the choice of the homotopy equivalence, but the images of $\mathscr{F}$ and $F$ coincide on the cohomology of $\X$. In particular, there exists an isomorphism of $\mathcal{O} H$-modules

\centers{$ \mathrm{H}^\bullet(\mathscr{C}_{(\lambda)}) \, \simeq \, \Hc^\bullet(\X,\mathcal{O})_{(\lambda)}$}

\noindent where the eigenspace on the right side is taken with respect to $F$. 

\sk

Moreover, if the terms of $\mathscr{C}$ are projective modules (for example if the action of $H$ is free) then the generalized $(\lambda)$-eigenspaces $\mathscr{C}_{(\lambda)}$ are in turn objects of the category $C^b(\mathcal{O}H$-$\mathrm{proj})$ and have, besides, the advantage of being in general much smaller than $\mathscr{C}$ itself.

\subsection{Finite reductive groups\label{1se3}}

\noindent \thesubsection.1. \textbf{Algebraic groups.} We keep the basic assumptions of the introduction, with some slight modification: $\G$ is a connected reductive algebraic group, together with an isogeny $F$, some power of which is a Frobenius endomorphism. In other words, there exists a positive integer $\delta$ such that $F^\delta$ defines a split $\mathbb{F}_{q^\delta}$-structure on $\G$ for a certain power $q^\delta$ of the characteristic $p$ (note that $q$ might not be an integer). We will assume that $\delta$ minimal for this property. For all  $F$-stable algebraic subgroup $\mathbf{H}$ of $\G$, we will denote by $H$ the finite group of fixed points $\mathbf{H}^F$.  \sk

We fix a Borel subgroup $\B$ containing a maximal torus $\T$ of $\G$ such that both $\B$ and $\T$ are $F$-stable. They define a root sytem $\Phi$ with basis $\Delta$, and a set of positive (resp. negative) roots $\Phi^+$ (resp. $\Phi^-$). Note that the corresponding Weyl group $W$ is endowed with a action of $F$, compatible with the isomorphism $W \simeq N_\G(\T)/\T$. Therefore, the image by $F$ of a root is a positive multiple of some other root, which will be denoted by $\phi^{-1}(\alpha)$, defining thus a bijection $\phi : \Phi \longrightarrow \Phi$. Since $\B$ is  $F$-stable, this map preserves $\Delta$ and $\Phi^+$. We will  use the notation $[\Delta/\phi]$ for a set of representatives of the orbits of $\phi$ on $\Delta$.

\mk

\noindent \thesubsection.2. \textbf{Deligne-Lusztig varieties.} Following \cite[Section 11.2]{BR1}, we fix a set of representatives $\{\dot w\}$ of $W$ in $N_\G(\T)$ and we define, for $w \in W$, the Deligne-Lusztig varieties $\X(w)$ and $\Y(\dot w)$ by:

\sk

\centers{$ \begin{psmatrix}[colsep=2mm,rowsep=10mm] \Y(\dot w) & = \, \big\{ g\U \in \G / \U \ \big| \ g^{-1}F(g) \in \U \dot w \U \big\} \\
					\X(w) & = \, \big\{ g\B \in \G / \B \ \big| \ g^{-1}F(g) \in \B w \B \big\} 
\psset{arrows=->>,nodesep=3pt} 
\everypsbox{\scriptstyle} 
\ncline{1,1}{2,1}<{\pi_w}>{/ \, \T^{wF}}		
\end{psmatrix}$}

\sk

\noindent where $\pi_w$ denotes the restriction to $\Y(\dot w)$ of the canonical projection $\G/\U \longrightarrow \G/\B$. They are both quasi-projective varieties endowed with a left action of $G$ by left multiplication. Furthermore, $\T^{wF}$ acts on the right of $\Y(\dot w)$ and $\pi_w$ is isomorphic to the corresponding quotient map, so that it induces a $G$-equivariant isomorphism of varieties $\Y(\dot w) / \T^{wF} \simeq \X(w)$. 

\sk

The $\ell$-adic cohomology of theses varieties yields the so-called Deligne-Lusztig induction.  More precisely, if $\theta$ is a character of $\T^{wF}$, one can look at the $\theta$-isotypic component of the cohomology and define the following virtual character

\centers{$\mathrm{R}_{w}(\theta) \, = \, \displaystyle \sum_{i\in \mathbb{Z}} (-1)^i\Hc^i(\Y(\dot w), K)_{\theta}.$}

\noindent Note that with our definition of the variety $\Y(\dot w)$ we have chosen to work with characters of $\T^{wF}$ instead of characters of $T_w$ for some torus $\T_w$ of type $w$. Our aim is to understand a far-reaching generalization of this character in the case where $w$ is a Coxeter element. It will be represented by a well-identified direct summand of the complex $\Rgc(\Y(\dot w),\Lambda)$.  

\section{The principal \texorpdfstring{$\ell$}{l}-block in the Coxeter case}

 In this preliminary section we introduce the main object of our study: the principal $\ell$-block $b$ of
 $G$ where the order of $q$ modulo $\ell$ is assumed to be the Coxeter number $h$. We will refer to this case as the \emph{Coxeter case}. This is in some sense the maximal interesting case in the modular representation theory of $G$, since $h$ is also the largest integer $d$ such that the cyclotomic polynomial $\Phi_d(q)$ divides the order of $G$.

\sk 

The results in \cite{BMM} express the irreducible characters of this block in terms of
irreducible components of Deligne-Lusztig characters $\mathrm{R}_{c}(\theta)$ induced from a Coxeter torus $\T^{cF}$. In this particular case, an explicit decomposition of these virtual characters is given by Lusztig's work on the cohomology of the Deligne-Lusztig variety $\X(c)$ \cite{Lu}. The characters of the block fall into two families:

\begin{itemize} 

\item[$\bullet$] the characters $\mathrm{R}_{c}(\theta)$ for $\theta$ a non-trivial $\ell$-character of the torus. These are irreducible characters (up to a sign);

\item[$\bullet$] the unipotent characters, coming from the cohomology of $\X(c)$.

\end{itemize}

\noindent Here the defect group of the principal $\ell$-block turns out to be a cyclic group and the distinction "non-unipotent/unipotent" translates into "exceptional/non-exceptional" in the theory of blocks with cyclic defect groups. The connection is actually much deeper: Hiss, L\"ubeck and Malle have observed in \cite{HLM} that the cohomology of the Deligne-Lusztig variety $\X(c)$ should not only give the characters of the block, but  also its Brauer tree. 

\sk 

We shall first review the geometric objects and the fundamental results involved in their description before recalling their conjecture.

\subsection{The Coxeter case\label{2se1}}

For the sake of simplicity, we first assume that $\G$ has no twisted components of type ${}^2$B$_2$, ${}^2$F$_4$ or ${}^2$G$_2$. The case of the Ree and Suzuki groups will be treated independently.

\mk 

\noindent \thesubsection.1. \textbf{Coxeter elements.} Let $V = X^\vee(\T) \otimes_\mathbb{Z} \mathbb{C}$ be
the $m$-dimensional vector space generated by the cocharacters of $\T$. The Weyl group $W$ can be
seen as a subgroup of the linear automorphisms of this vector space; moreover, the linear map
$\sigma = q^{-1} F$  has finite order $\delta$ and normalizes $W$ (with the previous assumptions on
$(\G,F)$, this is exactly the linear continuation of $\phi^\vee$). By \cite{St}, there exist eigenvectors
$(f_1,\ldots,f_m)$ of $\sigma$ in $S(V)$ of degrees $(d_1,\ldots,d_n)$ with associated eigenvalues
$(\varepsilon_1,\ldots,\varepsilon_m)$ such that $S(V)^W$ is isomorphic to the polynomial algebra
$\mathbb{C}[f_1,\ldots,f_m]$. Up to permutation, the pairs $(d_j,\varepsilon_j)$ are uniquely
determined by $\sigma$. The order of $G$ is then given by the following formula

\centers{$ |G| \, = \, \displaystyle q^N \prod_{j=1}^m (q^{d_j}-\varepsilon_j^{-1}) \, = \, q^N \, \prod_{d} \Phi_d(q)^{a(d)}$}

\noindent where $a(d)$ is the number of $j$ such that $\varepsilon_j = \exp(2 \mathrm{i} \pi d_j/d)$ \cite{BMa1}. The largest integer such that $a(d)$ is non-zero will be denoted by $h$ and referred as the \emph{Coxeter number} of the pair $(W,F)$.

\sk

From now on, we assume that $\G$ is semi-simple and $W$ is irreducible. In this case, the
$\mathbb{C}$-vector space $V = X^\vee(\T) \otimes_\mathbb{Z} \mathbb{C}$ can be identified with
the reflection representation of $W$ and the pairs $(d_j,\varepsilon_j)$ have been explicitely
computed in \cite{Car2}. From these values one easily deduces the Coxeter numbers for each type

\centers[3]{\begin{small} $ \begin{array}{c|c|c|c|c|c|c|c|c|c|c|c|c|c} \text{type} & \text{A}_n  & \text{B}_n & \text{D}_n & \text{E}_6 &  \text{E}_7 & \text{E}_8   & \text{F}_4 & \text{G}_2 & {}^2\text{A}_{2n} & {}^2\text{A}_{2n+1} & {}^2\text{D}_n & {}^3\text{D}_4 &  {}^2\text{E}_6 \\[3pt] \hline h \vphantom{\mathop{A}\limits^p}& n+1 & 2n & 2n-2 & 12 & 18 & 30 & 12 & 6  & 4n +2 & 4n +2 & 2n & 12 & 18\\ \end{array}$ \end{small}}      

\noindent and one can check that $a(h)$ is always equal to $1$.

\sk 

The twisted counterpart of the usual notion of Coxeter elements for Weyl groups has been introduced in \cite[Section 7]{Sp}:

\begin{de}\label{2de1}A \emph{Coxeter element} of the pair $(W,F)$ is a product $c= s_{\beta_1} \cdots s_{\beta_r}$ where $\{\beta_1,\ldots,\beta_r\} = [\Delta/\phi]$ is any set of representatives of the orbits of the simple roots under the action of $\phi$.
\end{de}

\noindent Such an element $c$ has the same properties as usual Coxeter elements, provided that the conjugation under $W$ is replaced by the $F$-conjugation. These properties become clear if we consider $c\sigma = q^{-1} cF \in \mathrm{GL}(V)$ instead of $c$:

\begin{thm}[Springer]\label{2thm1}Let $c$ be a Coxeter element of $(W,F)$ with $W$ irreducible.

\begin{itemize} %\setlength{\itemsep}{0pt}

 \item[$\mathrm{(i)}$] \vskip-4pt $c\sigma$ is an $h$-regular element.

 \item[$\mathrm{(ii)}$] Let $c'$ be any element of $W$. If  $c'\sigma$ has an eigenvalue of order $h$, it is $h$-regular and conjugated to $c\sigma$.  In particular, the set of Coxeter elements is contained in a single $F$-conjugacy class.

 \item[$\mathrm{(iii)}$] The eigenvalues of $c\sigma$ are $\varepsilon_j^{-1} \exp(2 \mathrm{i} \pi (d_j-1) /h)$. Moreover, the eigenvalues of order $h$ occur with multiplicity $1$. 

 \item[$\mathrm{(iv)}$] The centralizer $C_W(c\sigma)$ is a cyclic group generated by $(c\sigma)^\delta = c F(c) \cdots F^{\delta-1}(c)$.

\end{itemize}
\end{thm}

As a byproduct $\delta$ divides $h$. The quotient will be denoted by $h_0 = h/\delta$ in line with Lusztig's definition \cite[Section 1.13]{Lu}. For the sake of completeness, we give the different values of this number:

\centers[3]{\begin{small} $ \begin{array}{c|c|c|c|c|c|c|c|c|c|c|c|c|c} \text{type} & \text{A}_n  & \text{B}_n & \text{D}_n & \text{E}_6 &  \text{E}_7 & \text{E}_8   & \text{F}_4 & \text{G}_2 & {}^2\text{A}_{2n} & {}^2\text{A}_{2n+1} & {}^2\text{D}_n & {}^3\text{D}_4 &  {}^2\text{E}_6 \\[3pt] \hline h_0 \vphantom{\mathop{A}\limits^p}& n+1 & 2n & 2n-2 & 12 & 18 & 30 & 12 & 6  & 2n +1 & 2n +1 & n & 4 & 9\\ \end{array}$\end{small}}

\begin{rmk}\label{2rmk1}One could have defined the Coxeter number $h$ to be the maximal order of the eigenvalues of the elements of $W\sigma$. This is actually the original definition given by Springer \cite{Sp}, but it coincides with the previous one by the generic Sylow theorems \cite{BMa1}.
\end{rmk}

\mk

\noindent \thesubsection.2. \textbf{Coxeter tori.} Let $c$ be a Coxeter element of $(W,F)$. We will be interested
in rationnal tori $\T_c$ of type $c$, which are usually  called \emph{Coxeter tori}. Recall that
$(\T_c,F)$ is isomorphic to $(\T, cF)$ and that the order of the associated finite groups is given
by 

\centers{$ |T_c| \, = \, | \T^{cF}| \, = \, \det(qc\sigma - 1 \, | \, X^\vee(\T) \otimes_\mathbb{Z}
\mathbb{C}).$}

\noindent Since $a(h)=1$, the torus $\T_c$ contains a unique $\Phi_h$-Sylow subgroup $\mathbf{S}_h$
of $\G$, as defined in \cite{BMa1}. The following proposition summarizes the different properties we will use later on. They
are easily obtained by rephrasing Theorem \ref{2thm1} in the framework of \cite{BMa1} (see \cite{Du2} for more
details).

\begin{prop}\label{2prop1}Let $\ell$ be a prime number different from $p$. We assume that $\ell$ divides
$\Phi_h(q)$ but does not divide $|W^F|$. Then

\begin{itemize}

 \item[$\mathrm{(i)}$] The set $T_\ell$ of $\ell$-elements in $S_h$ is a cyclic $\ell$-Sylow
subgroup of $G$.

 \item[$\mathrm{(ii)}$] $C_\G(T_\ell) = \T_c$ and $N_G(T_\ell)/C_G(T_\ell) \simeq (N_\G (\T_c) /\T_c)^F \simeq
C_W(c\sigma)$.
 
 \item[$\mathrm{(iii)}$] Any non-trivial $\ell$-character $\theta$ of $T_c$ (or $\T^{cF}$) is in general position.
In other words, the centralizer $C_{C_W(c\sigma)}(\theta)$ is trivial. 

\end{itemize}
\end{prop}

\noindent \thesubsection.3. \textbf{The case of Ree and Suzuki groups.} The previous proposition holds also
when $\G$ has type ${}^2$B$_2$, ${}^2$F$_4$ or ${}^2$G$_2$. The notion of Coxeter elements has
indeed a natural generalization to these groups, taking into account that $\sigma = q^{-1} F$ does
no longer stabilize $X^\vee(\T)$ but only $X^\vee(\T) \otimes_\mathbb{Z} \mathbb{Z}[{\sqrt p}^{-1}]$
for $p=2$ or $3$ depending on the type of $(\G,F)$. The previous table can then be completed
with the orders of $c\sigma$:

\centers[1]{\begin{small}$  \begin{array}{c|c|c|c} \text{type} & {}^2\text{B}_2  &
{}^2\text{F}_4 & {}^2\text{G}_2
\\[3pt] \hline h \vphantom{\mathop{A}\limits^p} & 8 & 24 & 12 \\ \end{array}$\end{small}}   

\noindent Over the ring $\mathbb{Z}[\sqrt{p}]$, the cyclotomic polynomial $\Phi_h$ is no longer irreducible. For each type, the finite group $T_c$ itself is a cyclic group, and its order is given by the evalutation at $q$ of an irreducible factor of $\Phi_h$. When $q$ is positive, it is given by

\centers{\begin{small} $\begin{array}{c|c|c|c} \text{type} & {}^2\text{B}_2  &
{}^2\text{F}_4 & {}^2\text{G}_2
\\[3pt]\hline 
|T_c| \vphantom{\mathop{A}\limits^p} & 1-q\sqrt2 +q^2 & 1-q\sqrt2 + q^2 - q^3 \sqrt2
+q^4 & 1-q\sqrt3 + q^2 \\ \end{array}$\end{small}}  

\noindent Using  a case-by-case analysis, one can also check that when $\ell$ divides one of these
numbers without dividing the order of the corresponding Weyl group, the set of all $\ell$-elements
in $T_c$ is again a Sylow $\ell$-subgroup and it satisfies the assertions $\mathrm{(ii)}$ and
$\mathrm{(iii)}$ in Proposition \ref{2prop1}.

\subsection{Characters in the principal block\label{2se2}}

In order to use the results stated in the previous section, we will, until further notice, assume $\G$ to be  a
semi-simple group and $W$ to be irreducible (we shall say that $\G$ is \emph{quasi-simple}). We fix
a prime number $\ell$ not dividing the order of $W^F$ and satisfying one of the two following
assumptions, depending on the type of $(\G,F)$: 

\begin{itemize}
 
 \item[$\bullet$] "non-twisted" cases: $\ell$ divides $\Phi_h(q)$; 

 \item[$\bullet$] "twisted" cases: $\ell$ divides the order of $T_c$ for some Coxeter element $c$.

\end{itemize}

\noindent As in Section \ref{1se1}, the modular framework will be given by  an $\ell$-modular system $(K,\Lambda,k)$, which we require to be big enough for $G$. Note that the conditions on the prime number $\ell$ ensure that the class of $q$ in $k^\times$ is a primitive $h$-th root of unity. Indeed, $q^h$ is congruent to $1$ modulo $\ell$ and  for any proper divisor $m$ of $h$ the class of $q^m$ cannot be $1$ otherwise $1+q+\dots + q^{h-1}$ would be congruent to both $0$ and $h/m(1+ q + \cdots + q^{m-1})$. Taking $m$ to be minimal would force $\ell$ to divide $h/m$. The principal block of $\Lambda G$ for this particular class of primes will be at the center of our study.

\sk

We choose a Coxeter element $c$  of $(W,F)$ together with a maximal rational torus $\T_c$ of type $c$ (a Coxeter torus). This torus is the centralizer of a so-called $\Phi_h$-torus $\mathbf{S}_h$ (with $\mathbf{S}_h = \T_c$ for the Ree and Suzuki groups). As such, in the terminology of \cite{BMM}, it is a $h$-split Levi subgroup, and the $h$-cuspidal pair $(\T_c,1)$ corresponds to the principal $\ell$-block $b$ of $G$. More precisely, it follows from \cite[Theorem 5.24]{BMM} that the characters in $b$ are exactly the irreductible components of the Deligne-Lusztig characters $\mathrm{R}_{c}(\theta)$ where $\theta$ runs over the set of $\ell$-charaters of $\T^{cF}$. Two families of characters occur in this description; using Lusztig's results on the variety $\X(c)$, we now proceed with their parametrization.

\mk 

\noindent \thesubsection.1. \textbf{The non-unipotent characters in the block.} Proposition \ref{2prop1} says that any non-trivial $\ell$-character  of $\T^{cF}$ is in general position. Consequently, the corresponding induced character  is an actual irreducible character of $G$ (up to a sign). It is worth pointing out that this result is a consequence of a deep property of the cohomology of the Deligne-Lusztig variety $\Y(\dot c)$ \cite[Corollary 9.9]{DeLu}: for $\theta$ a non-trivial $\ell$-character of $\T^{cF}$, the $\theta$-isotypic component $\Hc^i(\Y(\dot c),K)_\theta$ of the cohomology in degree $i$ is non-zero for $i=\ell(c)=r$ only. 

\sk

The Frobenius endomorphism $F^\delta$ acts on $\T^{cF}$. Moreover, since $F^\delta$ acts trivially on $W$, the representative $\dot c$ of $c$ can be chosen to be $F^\delta$ -stable. In that case,  $\T^{cF} \rtimes \langle F^\delta \rangle_{\mathrm{mon}}$ acts on the variety $\Y(\dot c)$, leading to an linear action on the cohomology which satisfies

\centers{$ F^\delta \big(  \Hc^r(\Y( \dot c), K)_\theta \big) \, = \,  \Hc^r(\Y( \dot c), K)_{F^{\delta}(\theta)} \, = \,   \Hc^r(\Y( \dot c), K)_{v \cdot \theta} $}

\noindent where $v = c F(c) \cdots F^{\delta-1}(c)$. Note that $v$ is a generator (of order $h_0 = h/\delta$) of the cyclic group $C_W(c\sigma)$. Since the actions of $F^\delta$ and $G$ commute, the isotypic components associated to $\ell$-characters lying in the same orbit under $C_W(c\sigma)$ are isomorphic (via some power of $F^\delta$).

\begin{notation}\label{2not1}For $\theta$ a non-trivial $\ell$-character of $\T^{cF}$, the $\theta$-isotypic part of the bimodule $\Hc^r(\Y(\dot c),K)$ is a simple $KG$-module that will be denoted by $Y_\theta$. The isomorphic class of this module depends only on the orbit of $\theta$ under $C_W(c\sigma)$; the corresponding character will be denoted by $\chi_\theta$.
\end{notation}

\begin{prop}\label{2prop2}When $\theta$ runs over $\big[\Irr_\ell(\T^{cF}) / C_W(c\sigma)\big]$ and is assumed to be non-trivial, the characters $\chi_\theta$ are distinct irreducible cuspidal characters. Furthermore, they all have the same restriction to the set of $\ell$-regular elements of~$G$.
 \end{prop}

\begin{proof} The $F$-conjugacy class of a Coxeter element is cuspidal. In other words, $c$ is not contained in any  proper $F$-stable parabolic subgroup of $W$. An analogue property holds for the torus $\T_c$, and we deduce from \cite[Corollary 2.19]{Lu2} that the characters $\chi_\theta$ are cuspidal.  

\sk

Let $\theta$ and $\theta'$ be two non-trivial $\ell$-characters of $\T^{cF}$. The Mackey formula \cite[Theorem 11.13]{DM}, written for the torus $(\T,cF)$ instead of $(\T_c,F)$ yields

\centers{$ \langle \chi_\theta \, ; \, \chi_{\theta'} \rangle_G \, = \, \displaystyle \sum_{w \in W^{cF}} \langle \theta \, ; \, w\cdot \theta' \rangle_{\T^{cF}}.$}

\noindent Since both $\theta$ and $\theta'$ are in general position, we deduce that this sum is non-zero if and only if $\theta$ and $\theta'$ lie in the same orbit under $W^{cF} = C_W(c\sigma)$. 

\sk

Finally, the value on $\ell$-regular elements of any $\ell$-character is trivial. By the character formula \cite[Proposition 12.2]{DM}, it follows that the restriction of $\chi_\theta$ to the set of regular elements does not depend on $\theta$. \end{proof}

\noindent \thesubsection.2. \textbf{The unipotent characters in the block.} These are the irreducible components of the virtual character $\mathrm{R}_{c}(1)$ attached to the Deligne-Lusztig variety $\X(c)$. We review three main theorems in \cite{Lu} giving the fundamental properties of the cohomology of this variety, with a view to establish a simple parametrization of the unipotent characters of the principal $\ell$-block:

\begin{thm}[Lusztig]\label{2thm2}The Frobenius $F^\delta$ acts semi-simply on  $\bigoplus_i \Hc^i(\X(c),K)$ and its eigenspaces are mutually non-isomorphic simple $KG$-modules.
 \end{thm}

Moreover, Lusztig has shown that any eigenvalue of $F^\delta$ on the cohomology can be written $\zeta q^{m\delta/2}$, where $m$ is a non-negative integer and $\zeta$ is a root of unity. 
\pagebreak
These eigenvalues are explicitely determined in \cite[Table 7.3]{Lu}, and one can check the following numerical property by a case-by-case analysis:

\begin{fait}\label{2fact1}The $\ell$-reduction $\Lambda \twoheadrightarrow k$ induces a bijection between the eigenvalues of $F^\delta$ and the $h_0$-th roots of unity in $k$.
\end{fait}

Besides, the assumption on the prime number $\ell$ forces the $\ell$-reduction of $q^\delta$ to have order $h_0$ in $k^\times$, thus giving a canonical generator of the group of $h_0$-th roots of unity. We now choose a particular square root of $q^\delta$ in $K$ so that any eigenvalue of $F^\delta$ corresponds, via $\ell$-reduction, to a unique power of $q^\delta$. From this observation one can introduce the following notation:

\begin{notation}\label{2not2}For all $j=0,\ldots, h_0-1$, we denote by $\lambda_j$ the unique eigenvalue of $F^\delta$ on $\bigoplus_i \Hc^i(\X(c),K)$ which is congruent to $q^{j\delta}$ modulo $\ell$. Since a particular square root of $q^\delta$ has been chosen, there exists a unique root of unity $\zeta_j$ (in $K$) such that $\lambda_j = \zeta_j q^{m \delta/2}$ for some integer $m$. The eigenspace of $F^\delta$ associated to $\lambda_j$ will be denoted by $Y_j$ and its character by $\chi_j$.
\end{notation}

With this notation, the set $\{\chi_j \, | \, j=0, \ldots, h_0-1\}$ corresponds to the set of unipotent characters in the principal $\ell$-block. Note that it is important to keep track of the root of unity $\zeta_j \in K$ occurring in the eigenvalue $\lambda_j$. It gives indeed the Harish-Chandra series in which the corresponding eigenspace $Y_j$ falls:

\begin{thm}[Lusztig]\label{2thm3}The simple $KG$-modules $Y_i$ and $Y_j$ lie in the same Harish-Chandra series if and only if $\zeta_i = \zeta_j$.
\end{thm}

In other words, the set $\{Y_j \, | \, \zeta_j = \zeta\}$ represents the (possibly emply) intersection of the principal $\ell$-block with an Harish-Chandra series. The last result of this section tells us how these modules are precisely arranged in the cohomology of the Deligne-Lusztig variety $\X(c)$:

\begin{thm}\label{2thm4}Let $\zeta \in K$ be a root of unity. Assume that the set
of integers $j$ such that $\zeta_j = \zeta$ is non-empty. Then it is a set of consecutive integers $\intn{m_\zeta}{M_\zeta}$ and the corresponding eigenspaces of $F^\delta$ in the cohomology of $\X(c)$ are arranged as follows:

\centers{$ \begin{array}{c|c|@{\qquad \cdots \qquad}|c} 
\Hc^r(\X(c) ,K) & \Hc^{r+1}(\X(c) ,K)  & \Hc^{r+M_\zeta-m_\zeta}(\X(c) ,K) 
\\[4pt] \hline Y_{m_\zeta} \vphantom{\mathop{A}\limits^n} & Y_{m_\zeta+1} &  Y_{M_\zeta} \\ \end{array}$}

\noindent where $r = \ell(w)$. Moreover, $Y_j$ is a cuspidal $KG$-module if and only if $j = m_{\zeta_j} = M_{\zeta_j}$.

\end{thm}

More generally, the number $M_{\zeta_j} - m_{\zeta_j}$ measures the \emph{depth} of $Y_j$ as defined in \cite{Lu}, that is the obstruction of $Y_j$ from being cuspidal.

\subsection{Brauer tree of the principal block\label{2se3}}

The irreducible characters in the principal $\ell$-block split into two distinct families: $\big\{\chi_\theta \, | \, \theta \in [\mathrm{Irr}_\ell\T^{cF}/C_W(c\sigma)] \ \text{and} \ \theta \neq 1 \big\} $ and $\{\chi_i \, | \, i=0, \ldots, h_0 -1\}$. 
\pagebreak
The characters in the first set have the same restriction to $G_{\ell'}$ and as such, play the same role in the modular representation theory of $G$. If we define the \emph{exceptional character} $\chi_\mathrm{exc}$ to be the sum of these elements, then the structure of the block can be expressed in terms of the elements of the set $\mathscr{V} = \{\chi_0, \chi_1, \ldots, \chi_{{h_0}-1}\}\cup \{ \chi_\mathrm{exc}\}$. More precisely, from the theory of blocks with cyclic defect groups one knows that the character of any indecomposable projective $\Lambda G$-module can be written as $[P] = \chi + \chi'$ where $\chi$ and $\chi'$ are two distinct elements of $\mathscr{V}$.  Following Brauer, one can define a graph $\Gamma$ encoding the structure of the block:

\begin{itemize}

\item[$\bullet$] the vertices of the graph are labeled by $\mathscr{V} = \{\chi_0, \chi_1, \ldots, \chi_{{h_0}-1}\}\cup \{ \chi_\mathrm{exc}\}$. The vertex associated to $\chi_\mathrm{exc}$ is called  the \emph{exceptional vertex} or the \emph{exceptional node}. By extension, the other vertices are said to be  \emph{non-exceptional};

\item[$\bullet$] two vertices $\chi$ and $\chi'$ are connected by an edge if and only if $\chi + \chi'$ is the character of an indecomposable projective $\Lambda G$-module.

\end{itemize}

\noindent This graph $\Gamma$ is actually a tree, which we refer as the \emph{Brauer tree} of the block. The edges of the tree can be labeled either by the indecomposable projective $\Lambda G$-modules in the block or by the simple $kG$-modules in the block (for the indecomposable projective modules are exactly the projective covers of the simple modules).

\begin{exemple}\label{2ex1}Let $D$ be a cyclic $\ell$-group and $E$ an $\ell'$-subgroup of $\mathrm{Aut}(D)$. The Brauer tree of the unique $\ell$-block of the group $H = D \rtimes E$ is a star. Indeed, the simple $kH$-modules are obtained by  inflation of the simple $kE$-modules (with a trivial action of $D$). Since $E$ is an $\ell'$-group, such a module lifts uniquely to a $\Lambda E$-lattice $\widetilde S$ with projective cover $\mathrm{Ind}_E^H \widetilde S$. The character of the latter decomposes into

\centers{$ [\mathrm{Ind}_{E}^{H} \, \widetilde{S}] \, = \, [\widetilde{S}] + \chi_\mathrm{exc}$}

\noindent where $\chi_\mathrm{exc}$ is the sum of the irreducible ordinary characters of $H$ with trivial restriction to $D$ (meaning that the restriction to $D$ comes from a trivial representation of $D$). By construction, the Brauer tree has the following shape:

\sk

\begin{figure}[h] 
\centers[3]{\begin{pspicture}(4,4)
  \psset{linewidth=1pt}

  \cnode[fillstyle=solid,fillcolor=black](2,2){5pt}{A2}
    \cnode(2,2){8pt}{A}
  \cnode(4,2){5pt}{B}
  \cnode(3.73,3){5pt}{C}
  \cnode(3.73,1){5pt}{D}
    \cnode(3,3.73){5pt}{H}
    \cnode(3,0.27){5pt}{I}
    \cnode(2,4){5pt}{J}

  \cnode(0,2){5pt}{E}
  \cnode(0.27,3){5pt}{F}
  \cnode(0.27,1){5pt}{G}

  \ncline[nodesep=0pt]{A}{B}
  \ncline[nodesep=0pt]{A}{C}
  \ncline[nodesep=0pt]{A}{D}
  \ncline[nodesep=0pt]{A}{E}
  \ncline[nodesep=0pt]{A}{F}
  \ncline[nodesep=0pt]{A}{G}
  \ncline[nodesep=0pt]{A}{H}
  \ncline[nodesep=0pt]{A}{I}
  \ncline[nodesep=0pt]{A}{J}

\psellipticarc[linestyle=dotted,linewidth=1.5pt](2,2)(2,2){103}{135}
\psellipticarc[linestyle=dotted,linewidth=1.5pt](2,2)(2,2){225}{287}

%  \ncline[nodesep=0pt]{C}{D}\naput[npos=-0.15]{$t_3$} \naput[npos=1.15]{$t_4$}
 % \ncline[nodesep=0pt,linestyle=dotted]{D}{E}
 % \ncline[nodesep=0pt]{E}{F}\naput[npos=-0.15]{$t_{n-1}$} \naput[npos=1.15]{$t_n$}

\end{pspicture}}
\caption{Brauer tree of $D \rtimes E$}
\end{figure}
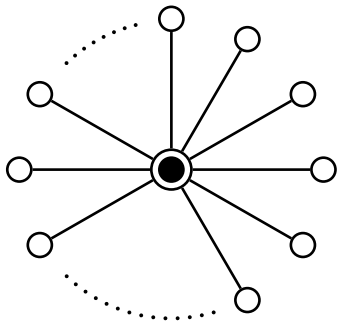

\noindent Rickard \cite{Ri}, \cite{Ri3} and Linckelmann \cite{Lin1} have shown that one can "unfold" the Brauer tree of the principal $\ell$-block of $G$ in order to obtain this star, and that this "unfolding" is the reflect of a splendid equivalence between the principal $\ell$-blocks of $G$ and $\T^{cF} \rtimes C_W(c\sigma)$. One of the main goals of this article is to show that this operation can be performed by means of the complex $\Rgc(\Y(\dot c),\Lambda)$ representing the cohomology of the Deligne-Lusztig variety $\Y(\dot c)$, thus giving a geometric explanation of Rickard-Linckelmann's result for finite reductive groups.

\end{exemple}

The conjecture of Hiss, L\"ubeck and Malle stated in \cite{HLM} comes within this geometric framework. It predicts the shape of the Brauer tree of the principal $\ell$-block of $G$ in terms of the parametrization of the unipotent characters given previously:

\begin{conjHLM}[Hiss-L\"ubeck-Malle]\label{2conj1}Let $\Gamma^\bullet$  denote the graph obtained from the Brauer tree of the principal $\ell$-block by removing the exceptional node and all edges incident to it. Then the following holds:

\begin{enumerate} 

\item[$\mathrm{(i)}$]  The connected components of $\Gamma^\bullet$ are labeled by the Harish-Chandra series, hence by the roots of unity $\zeta_j$'s.

\item[$\mathrm{(ii)}$]  The connected component corresponding to a root $\zeta$ is:

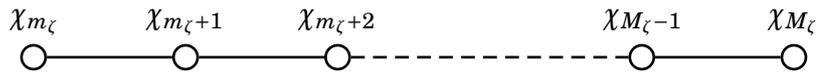
\centers[0]{ \begin{pspicture}(10,1)
  \psset{linewidth=1pt}

  \cnode(0,0.2){5pt}{A}
  \cnode(2,0.2){5pt}{B}
  \cnode(4,0.2){5pt}{C}
  \cnode(8,0.2){5pt}{D}
  \cnode(10,0.2){5pt}{E}
  \ncline{A}{B}  \naput[npos=-0.1]{$\vphantom{\Big(}\chi_{m_\zeta}$} \naput[npos=1.1]{$\vphantom{\Big(}\chi_{m_\zeta+1}$}
  \ncline{B}{C}  \naput[npos=1.1]{$\vphantom{\Big(}\chi_{m_\zeta+2}$}
  \ncline[linestyle=dashed]{C}{D}
  \ncline{D}{E}  \naput[npos=-0.1]{$\vphantom{\Big(}\chi_{M_\zeta-1}$} \naput[npos=1.1]{$\vphantom{\Big(}\chi_{M_\zeta}$}
\end{pspicture}}

\item[$\mathrm{(iii)}$] The vertices labeled by $\chi_{m_{\zeta}}$ are the only nodes connected to the exceptional node.
\end{enumerate}

\end{conjHLM}

The validity of this conjecture has been checked in all cases where the Brauer tree was known, that is for all quasi-simple groups except the groups of type E$_7$ and E$_8$. A  general proof will be exposed  in the next section, but under a precise assumption on the torsion in the cohomology of $\X(c)$ (see the introduction or section \ref{3se2} for more details).

\sk

As defined above, the Brauer tree of a block encodes only its decomposition matrix. However, as one can notice in the previous example, once the nodes have been labeled, there are different ways to draw the star. The planar embedding of the Brauer tree is actually a fundamental datum of the block: consider a vertex in the tree, together with all the edges incident to it, labeled by the simple $kG$-modules $S_1, \ldots, S_m$. Then from the general theory there exist uniserial modules which have exactly the $S_i$'s as composition factors. The unique composition series of any  of these determines an ordering of $S_1,\ldots, S_m$ which, up to cyclic permutation, does not depend on the module \cite{Gr}. In the \emph{planar embedded}  Brauer tree, the edges incident to that vertex  are labeled anti-clockwise according to this ordering. For a description of this ordering in terms of extentions, one can readily check that two edges labeled by the simple modules $S$ and $S'$ are (anti-clockwise) adjacent if and only if $\mathrm{Ext}_{kG}^1(S',S) \neq 0$.

\sk

The Brauer tree, together with its planar embedding, fully encodes the representation theory of the block since it determines the block algebra up to Morita equivalence.

\begin{exemple}\label{2ex2}We return to the previous example in the special case where  $E$ is a cyclic group, generated by an element $x \in \mathrm{Aut}(D)$ of order $m$ prime to $\ell$. Recall that $D$ is also assumed to be cyclic; consequently, there exists an integer $n$ prime to $|D|=\ell^\alpha$, uniquely determined in $\intn{0}{\ell^\alpha -1}$ such that $x$ acts on any element $y \in D$ by raising $y$ to the power of $n$. Since $x$ has order $m$, the $\ell$-reduction of $n$ has order $d | m$. Besides, $\ell > m$ so that 
$v_\ell(n^d-1) = v_\ell(n^m -1) \geq \ell^\alpha$, which forces $d=m$. 

\sk

By Hensel's lemma, there exists a (unique) primitive $m$-th root of unity $\zeta \in \mathbb{Z}_\ell^\times$ such that $\zeta \equiv n \ \mathrm{mod} \ \ell \mathbb{Z}_\ell$. If we number the exceptional characters $\eta_j : H \longrightarrow \mathbb{Z}_\ell^\times$ in such a way that $\eta_j(x) = \zeta^j$, then the planar embedded Brauer tree is given by

\begin{figure}[h]
\centers[0]{\begin{pspicture}(4,4.5)

  \psset{linewidth=1pt}

   \cnode[fillstyle=solid,fillcolor=black](2,2){5pt}{A2}
       \cnode(2,2){8pt}{A}
  \cnode(4,2){5pt}{B}
  \cnode(3.73,3){5pt}{C}
  \cnode(3.73,1){5pt}{D}
    \cnode(3,3.73){5pt}{H}
    \cnode(3,0.27){5pt}{I}
    \cnode(2,4){5pt}{J}

  \cnode(0,2){5pt}{E}
  \cnode(0.27,3){5pt}{F}
  \cnode(0.27,1){5pt}{G}

  \ncline[nodesep=0pt]{A}{B}\ncput[npos=1.4]{$\eta_0$}
  \ncline[nodesep=0pt]{A}{C}\ncput[npos=1.45]{$\eta_1$}
  \ncline[nodesep=0pt]{A}{D}\ncput[npos=1.38]{$\hphantom{aaa}\eta_{m-1}$}
  \ncline[nodesep=0pt]{A}{E}
  \ncline[nodesep=0pt]{A}{F}
  \ncline[nodesep=0pt]{A}{G}
  \ncline[nodesep=0pt]{A}{H}\ncput[npos=1.42]{$\eta_2$}
  \ncline[nodesep=0pt]{A}{I}\ncput[npos=1.4]{$\hphantom{aaa}\eta_{m-2}$}
  \ncline[nodesep=0pt]{A}{J}\ncput[npos=1.4]{$\eta_3$}

\psellipticarc[linestyle=dotted,linewidth=1.5pt](2,2)(2,2){103}{135}
\psellipticarc[linestyle=dotted,linewidth=1.5pt](2,2)(2,2){225}{287}
\psellipticarc[linewidth=1pt]{->}(2,2)(1,1){103}{135}
\psellipticarc[linewidth=1pt]{->}(2,2)(1,1){225}{287}

\end{pspicture}}
\caption{Planar embedded Brauer tree of $D \rtimes E$}
\end{figure}
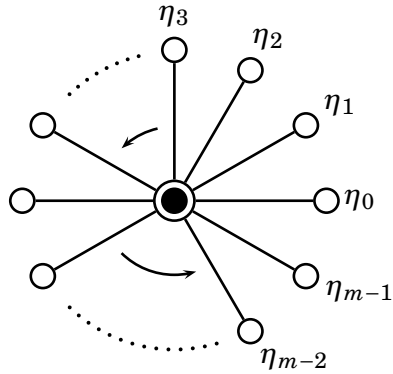
\mk

\noindent  In other words, if we denote by $k_j$ the simple (one-dimensional) $k H$-module attached to the ordinary character $ \eta_j$, then $\mathrm{Ext}_{kH}^1(k_i,k_j)$ is non-trivial if and only if  $i \equiv j+1\mathrm{ \ mod\ } m$.
 \end{exemple}

The previous example has been thoughtfully chosen, since it gives also the planar embedded Brauer tree of the principal $\ell$-block of $H = \T^{cF} \rtimes C_W(c\sigma)$. In this case, one should notice that the generator $x = (cF(c) \cdots F^{\delta-1}(c))^{-1}$ of $C_W(cF)$ acts on $\T^{cF}$ via the Frobenius $F^\delta$, thus raising  any element to the power of $q^\delta$.   On the other side, the non-exceptional characters in the principal $\ell$-block of $G$ are labeled by the eigenvalues of $F^\delta$ on the cohomology of $\X(c)$.  Since the two actions of $F^\delta$ are compatible and give the same parametrization of the characters, it seems reasonable to complete Conjecture \hyperref[2conj1]{(HLM)} with

\begin{conjHLM2}\label{2conj2}The $\ell$-reduction of $q^\delta$ gives a canonical generator of the group of $h_0$-th roots of unity and the corresponding order induces the cyclic ordering around the exceptional node of the Brauer tree of the principal $\ell$-block of~$G$.
\end{conjHLM2}

Such a result is of course only interesting for trees in which different planar embeddings are possible. This is the case for groups of type F$_4$, ${}^2$F$_4$, E$_7$, E$_8$ and ${}^2$G$_2$ only. For the latter, according to Conjectures \hyperref[2conj1]{(HLM)} and  \hyperref[2conj2]{(HLM+)} the Brauer tree should be given by Figure \ref{fig3}, where $\mathrm{i} = \xi^3$ and $\xi$ is the unique $12$th root of unity in $\Lambda^\times$ congruent to $q^5$ modulo $\ell$. The ordering becomes clear if we choose to label the non-exceptional vertices by the integers  $q^{j\delta}$ congruent to the eigenvalues of $F^\delta$.

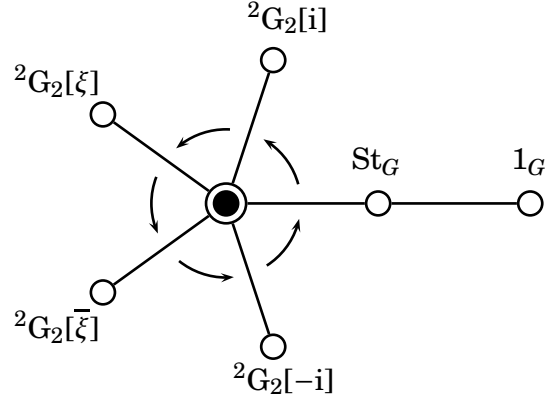
\begin{figure}[h]
\centers[3]{\begin{pspicture}(6,4.5)

  \psset{linewidth=1pt}

  \cnode[fillstyle=solid,fillcolor=black](2,2){5pt}{A2}
  \cnode(2,2){8pt}{A}
  \cnode(4,2){5pt}{B}
  \cnode(6,2){5pt}{C}
  \cnode(2.62,3.9){5pt}{D}
  \cnode(2.62,0.1){5pt}{E}
  \cnode(0.38,3.18){5pt}{F}
  \cnode(0.38,0.82){5pt}{G}

  \ncline[nodesep=0pt]{A}{B}\naput[npos=1.1]{$\vphantom{\Big(} \mathrm{St}_G$}
  \ncline[nodesep=0pt]{B}{C}\naput[npos=1.1]{$\vphantom{\Big(} \mathrm{1}_G$}
  \ncline[nodesep=0pt]{A}{D}\ncput[npos=1.5]{${}^2\text{G}_2[\mathrm{i}]$}
  \ncline[nodesep=0pt]{A}{E}\ncput[npos=1.4]{${}^2\text{G}_2[-\mathrm{i}]$}
  \ncline[nodesep=0pt]{A}{F}\ncput[npos=1.6]{${}^2\text{G}_2[{\xi}]$}
  \ncline[nodesep=0pt]{A}{G}\ncput[npos=1.6]{${}^2\text{G}_2[\overline{\xi}]$}
  \psellipticarc[linewidth=1pt]{->}(2,2)(1,1){15}{60}
  \psellipticarc[linewidth=1pt]{->}(2,2)(1,1){87}{132}
  \psellipticarc[linewidth=1pt]{->}(2,2)(1,1){159}{204}
  \psellipticarc[linewidth=1pt]{->}(2,2)(1,1){231}{276}
  \psellipticarc[linewidth=1pt]{->}(2,2)(1,1){303}{348}

\end{pspicture}}
\caption{Brauer tree for a group of type ${}^2$G$_2$} \label{fig3}
\end{figure}
\bk

%\begin{figure}
%\centers[3]{\begin{pspicture}(6,4.2)

  %\psset{linewidth=1pt}

  %\cnode[fillstyle=solid,fillcolor=black](2,2){5pt}{A2}
  %\cnode(2,2){8pt}{A}
  %\cnode(4,2){5pt}{B}
  %\cnode(6,2){5pt}{C}
  %\cnode(2.62,3.9){5pt}{D}
  %\cnode(2.62,0.1){5pt}{E}
  %\cnode(0.38,3.18){5pt}{F}
  %\cnode(0.38,0.82){5pt}{G}

  %\ncline[nodesep=0pt]{A}{B}\naput[npos=1.1]{$\vphantom{\big(} 1$}
  %\ncline[nodesep=0pt]{B}{C}\naput[npos=1.1]{$\vphantom{\big(} q^2$}
  %\ncline[nodesep=0pt]{A}{D}\ncput[npos=1.45]{$q^4$}
  %\ncline[nodesep=0pt]{A}{E}\ncput[npos=1.4]{$q^{10}$}
 % \ncline[nodesep=0pt]{A}{F}\ncput[npos=1.45]{$q^6$}
  %\ncline[nodesep=0pt]{A}{G}\ncput[npos=1.45]{$q^8$}
  %\psellipticarc[linewidth=1pt]{->}(2,2)(1,1){15}{60}
%  \psellipticarc[linewidth=1pt]{->}(2,2)(1,1){87}{132}
  %\psellipticarc[linewidth=1pt]{->}(2,2)(1,1){159}{204}
  %\psellipticarc[linewidth=1pt]{->}(2,2)(1,1){231}{276}
  %\psellipticarc[linewidth=1pt]{->}(2,2)(1,1){303}{348}
%%\pscurve[linecolor=red]{->}(2.5,1.7)(4.5,1.7)(6.5,1.7)(7,2)(6.5,3)(4.5,3)

%\end{pspicture}}
%\caption{Brauer tree for a group of type ${}^2$G$_2$ with eigenvalues of $F^2$}
%\end{figure}
%\bk

\sk

We will give in the last section a proof of this conjecture under the assumption  $\mathrm{(S)}$ (see the introduction).

\section{On the Hiss-L\"ubeck-Malle conjecture}

The aim of this section is to present a general proof of Conjecture \hyperref[2conj1]{(HLM)} under a precise assumption on the torsion in the cohomology. We follow here the geometric approach suggested in \cite{HLM}. The geometry of the Deligne-Lusztig varieties $\X(c)$ and $\Y(\dot c)$, and especially their cohomology with coefficients in $\Lambda$, should contain the information needed to understand  the structure of the principal $\ell$-block of $G$. The fundamental work of Lusztig on these varieties \cite{Lu} will be the starting point for our proof.

\sk

The first part of the proof deals with the non-cuspidal $kG$-modules and their projective cover. We show that the contribution of these modules to the Brauer tree of the principal $\ell$-block consists of the lines represented in the second assertion of the conjecture. The final part of the proof is obtained by showing that the edges labeled by the cuspidal modules are exactly the edges incident to the exceptional node. This is where the assumption $\mathrm{(W)}$ comes in, since it allows us, with the help of the tools introduced in Section \ref{1se2}, to single out projective modules with character $\chi_\mathrm{exc} + \chi_{m_\zeta}$ in the cohomology of $\Y(\dot c)$, thus giving the missing edges.

\subsection{Non-cuspidal \texorpdfstring{$kG$}{kG}-modules\label{3se1}}

Let $I \subset \Delta$ be a $\phi$-stable subset of simple roots. The corresponding standard parabolic subgroup $\P_I$ and Levi complement $\L_I$ are both $F$-stable. Recall  that the Harish-Chandra induction and restriction functors are defined over any coefficient ring $\mathcal{O}$ among $(K,\Lambda,k)$ by

\leftcenters{and}{$ \begin{array}[b]{l} \hphantom{{}^*}\mathrm{R}_{L_I}^G \, : \begin{array}[t]{rcl} \mathcal{O} L_I \text{-}\mathrm{mod} & \longrightarrow & \mathcal{O} G \text{-}\mathrm{mod}  \\ 
N & \longmapsto & \mathcal{O}[G/U_I] \otimes_{\mathcal{O} L_I} N \end{array} \\[20pt]
 {}^*\mathrm{R}_{L_I}^G \, : \begin{array}[t]{rcl} \mathcal{O} G \text{-}\mathrm{mod} & \longrightarrow & \mathcal{O} L_I \text{-}\mathrm{mod}  \\ 
M & \longmapsto & M^{U_I}. \end{array} \end{array}$}

\noindent We give the basic properties of these functors that we shall use in this section (see for example \cite[Section 3.A]{Ge3}):

\begin{itemize}

\item $\mathrm{R}_{L_I}^G$ and ${}^*\mathrm{R}_{L_I}^G$ are exact functors. They stabilize the categories of finitely generated projective modules and finitely generated $\mathcal{O}$-free modules. 

\item $\mathrm{R}_{L_I}^G$ and ${}^*\mathrm{R}_{L_I}^G$ are compatible with scalar extension  $- \otimes_\Lambda K$ and $\ell$-reduction $- \otimes_\Lambda k$.

\end{itemize}

\noindent In particular, the induced morphisms between the Grothendieck groups are compatible with the decomposition maps. More precisely, if we denote by $\mathrm{dec}_G$ (resp. $\mathrm{dec}_{L_I}$) the decomposition map between $K_0(KG$-$\mathrm{mod})$ and $K_0(kG$-$\mathrm{mod})$  (resp. between $K_0(KL_I$-$\mathrm{mod})$ and $K_0(kL_I$-$\mathrm{mod})$), then

\centers{$ \dec_G \circ \mathrm{R}_{L_I}^G \, = \, \mathrm{R}_{L_I}^G \circ \dec_{L_I}\qquad \text{and} \qquad \dec_{L_I} \circ {}^*\mathrm{R}_{L_I}^G \, = \, {}^*\mathrm{R}_{L_I}^G \circ \dec_{G}.$}

We shall first study the projective covers of the non-cuspidal simple $kG$-modules using their Harish-Chandra restriction. This method relies on the fact that the restriction of the cohomology of $\X(c)$ can be expressed in terms of Coxeter varieties associated to "smaller" groups, for which the module categories over $k$ are semi-simple.

\bk

\noindent \thesubsection.1. \textbf{Restriction of a $kG$-module.} From now on, the Deligne-Lusztig variety associated to the Coxeter element $c$ will be simply denoted by $\X$. If $I$ is any $\phi$-stable subset of $\Delta$, one obtains a Coxeter element $c_I$ of $(W_I,F)$ by removing from $c$ the reflexions associated to the simple roots which are not in $I$. Written with the Borel subgroup  $\B_I = \B \cap \L_I$ of $\L_I$, the Deligne-Lusztig variety $\X_I$ associated to $c_I$ is by definition

 \centers{$ \X_I \, = \, \X_{\L_I}(c_I) \, = \, \big\{g\B_I \in \L_I/ \B_I\, \big| \, g^{-1}F(g) \in \B_I c_I \B_I \big\}.$}

\noindent When $I$ is a maximal proper subset of $\Delta$, the cohomology groups of $\X$ and $\X_I$ are closely related (see \cite[Corollary 2.10]{Lu}):

\begin{prop}\label{3prop1}Assume that $I$ is a maximal proper $\phi$-stable subset of $\Delta$. Then there exists an isomorphism of $KL_I$-modules, compatible with the action of $F^\delta$

\centers{${}^*\mathrm{R}_{L_I}^G \big(\Hc^i(\X,K)\big) \, \simeq \, \Hc^{i-1}(\X_I,K) \oplus \Hc^{i-2}(\X_I,K)(-1)$}

\noindent where the symbol $(-1)$ indicates a Tate twist. 
\end{prop}

By successive applications of this proposition, one can extend the previous isomorphism to the case where $I$ is not assumed to be maximal:

\centers{${}^*\mathrm{R}_{L_I}^G \big(\Hc^i(\X,K)\big) \, \simeq \, \displaystyle \bigoplus_{j \geq 0} \binom{r-r_I}{j} \Hc^{i-r+r_I-j}(\X_I,K) (-j) $}

\noindent with $r_I = |I/\phi|$ the number of $\phi$-orbits in $I$. In particular, any eigenvalue of $F^\delta$ on the cohomology of $\X_I$ is also an eigenvalue of $F^\delta$ as an endomorphism of the cohomology of $\X$. In order to rewrite the previous proposition in terms of eigenspaces, we introduce the following notation, valid for any $\phi$-stable subset $I$ of $\Delta$:

\begin{notation}\label{3not1}Let $j \in \intn{0}{h_0-1}$ and $\lambda_j = \zeta_j q^{m\delta /2}$ the corresponding eigenvalue of $F^\delta$ on $\Hc^\bullet(\X,K)$. We will denote by $Y_j^I$ the $\lambda$-eigenspace of $F^\delta$ on $\Hc^\bullet(\X_I,K)$. It is a $KL_I$-module (which can be zero) and we will denote by $\chi_j^I$ its character.
\end{notation}

\begin{rmk}\label{3rmk1}As soon as the order of $F$ on $W_I$ is also equal to $\delta$, Theorem \ref{2thm2} applies to the variety $\X_I$, even if $W_I$ is not irreductible or $\L_I$ is not semi-simple. The non-zero $Y_j^I$, which are exactly the eigenspaces of $F^\delta$ on $\Hc^\bullet(\X_I,K)$ are then mutually non-isomorphic simple $KL_I$-modules. The only non-trivial cases where the order of $F$ on $W_I$ is strictly lower than $\delta$ arise when $W_I$ is split of type A$_n$. But it that special case, the eigenvalues of a Frobenius endomorphism are $1,q,\ldots, q^n$ so that there cannot be any difference between the eigenspaces of $F$ and $F^\delta$. Consequently, the non-zero modules $Y_j^I$ are eigenspaces of $F$ and as such, are still simple and mutually non-isomorphic.
\end{rmk}

Using this notation, the previous proposition can be rephrased in terms of restriction of simple modules. Indeed, when $I$ is assumed to be maximal, the restriction of a simple unipotent $KG$-module of the principal $\ell$-block is given by
 
 \begin{equation}\label{3eq1}{}^*\mathrm{R}_{L_I}^G (Y_j) \, \simeq \, Y_j^{I} \oplus Y_{j-1}^I. \end{equation}

\bk

\noindent \thesubsection.2. \textbf{Finding the non-cuspidal simple $kG$-modules.}  Every simple $kG$-module $M$ has a projective cover, which lifts uniquely (up to isomorphism) to an indecomposable projective $\Lambda G$-module that we will denote by $P_M$. The following lemma gives the character of this module in the case where $M$ is non-cuspidal:

\begin{lem}\label{3lem1}Let $M$ be a simple $kG$-module in the principal $\ell$-block. Assume that $M$ is non-cuspidal. Then there exists a unique integer $j \in \intn{0}{h_0-1}$ with $j > m_{\zeta_j}$ such that

\centers{$ [P_M] \, = \, [Y_j] + [Y_{j-1}] \, = \, \chi_j + \chi_{j-1}.$}

\noindent Moreover, if $I$ is any maximal proper $\phi$-stable subset of $\Delta$, then the restriction ${}^*\mathrm{R}_{L_I}^G(M)$ lifts uniquely (up to isomorphism) to a $\Lambda L_I$-lattice with character $\chi_j^I$.
\end{lem}

\begin{proof} Recall that ${}^*\mathrm{R}_{L_I}^G$ is compatible with $\ell$-reduction, so that the composition factors of any $\ell$-reduction of a cuspidal $KG$-module are cuspidal. By Proposition \ref{2prop2}, $\chi_\mathrm{exc}$ is a cuspidal character and as such cannot be a component of the character of $P_M$. From the general theory of blocks with cyclic defect groups (see Section \ref{2se3}) we deduce that there exist two distinct integers $i,j \in \intn{0}{h_0-1}$ such that

\centers{$ [P_M] \, = \, [Y_j] + [Y_i].$}

\noindent In other words, $M$ is a composition factors of the $\ell$-reductions of both $Y_j$ and $Y_i$.

\sk

Let $I$ be a maximal proper $\phi$-stable subset of $\Delta$, such that the restriction ${}^*\mathrm{R}_{L_I}^G(M)$ is non-zero (such a subset always exists since $M$ is non-cuspidal). By the various properties of the restriction functors (listed at the beginning of the section), the composition factors of ${}^*\mathrm{R}_{L_I}^G(M)$ occur as composition factors of the $\ell$-reductions of both  ${}^* \mathrm{R}_{L_I}^G(Y_j)$ and ${}^* \mathrm{R}_{L_I}^G(Y_i)$. Besides, by Formula \ref{3eq1}, these modules decomposes into

\leftcenters{and}{$ \begin{array}[b]{r@{\, \ \simeq \ \, }l}{}^* \mathrm{R}_{L_I}^G(Y_j) &  Y_j^I \oplus Y_{j-1}^I
\\[5pt] {}^*\mathrm{R}_{L_I}^G(Y_i)&  Y_i^I \oplus Y_{i-1}^I. \end{array}$}

\noindent Since $\L_I$ is a proper Levi subgroup of $\G$, the associated finite group has order prime to $\ell$ (otherwise $\L_I$ would contain a Coxeter torus, which is impossible by the results in Section \ref{2se1}). Therefore, the $KL_I$-modules remain simple after $\ell$-reduction so that the modules $ {}^*\mathrm{R}_{L_I}^G(Y_i)$ and $ {}^*\mathrm{R}_{L_I}^G(Y_j)$ must have a common irreducible component. By Remark \ref{3rmk1}, this forces  $|i-j|=1$, $\zeta_i=\zeta_j$ and ${}^* \mathrm{R}_{L_I}^G(M)$ to be a simple $kL_I$-module, isomorphic to any $\ell$-reduction of $Y_{\mathrm{min}(i,j)}^I$. 
\end{proof}

\begin{notation}\label{3not2}Up to isomorphism, for any $j$ there exists at most one module $M$ satisfying the properties of the previous lemma. Such a module will be denoted by $S_j$. If it does not exist or if it is cuspidal, we will set $S_j = \{0\}$ by convention.
\end{notation}

With this notation, Lemma \ref{3lem1} asserts that the non-cuspidal composition factors of any $\ell$-reduction of $Y_j$ are isomorphic to $S_j$ or $S_{j+1}$. This gives the non-cuspidal part of the Brauer tree. It remains to determine whether $S_j$ is always non-zero:

\begin{lem}\label{3lem2}Let $j \in \intn{0}{h_0-1}$ and assume  that $j > m_{\zeta_j}$. Then $S_j$ is non-zero and occurs with multiplicity one as a composition factor of the $\ell$-reductions of both $Y_j$ and $Y_{j-1}$.
\end{lem}

\begin{proof} Our assumption on $j$ implies that $Y_j$ occurs as an eigenspace of $F^\delta$ on some cohomology group $\Hc^i(\X,K)$ of degree  $i> r$. It is also non-cuspidal, and we can choose a maximal proper $\phi$-stable subset $I$ of simple roots such that the restriction ${}^*\mathrm{R}_{L_I}^G\big(Y_j)$ is non-zero. By Formula \ref{3eq1}, this forces one of the modules between $Y_j^I$ and $Y_{j-1}^I$ to be non-zero. The latter is actually always non-zero: indeed, by Proposition \ref{3prop1}, the module $Y_{j}^I$ is an eigenspace of $F^\delta$ on $\Hc^{i-1}(\X_I,K)$ with $i-1 > r-1$. But the Coxeter variety $\X_I$ has dimension $r-1$, and Theorem \ref{2thm4} ensures that $Y_{j-1}^I$ is non-zero as soon as $Y_j^I$ is.

\sk

Denote by $a$ (resp. $b$)  the multiplicity of $S_j$ (resp. $S_{j+1}$) in the composition series of the $\ell$-reductions of $Y_j$ (by convention, the multiplicity is set to zero if the module $S_m$ is zero). By the previous lemma, these modules are the only possible non-cuspidal composition factors. Since the restriction functor is compatible with the decomposition maps, we can write

\centers{$ \dec_{L_I}\big(\big[{}^* \mathrm{R}_{L_I}^G (Y_j)\big]\big) \, = \, a\big[{}^* \mathrm{R}_{L_I}^G (S_j)\big]_k + b \big[{}^* \mathrm{R}_{L_I}^G (S_{j+1})\big]_k$}

\noindent in $K_0(kG$-$\mathrm{mod})$. Together with Formula \ref{3eq1}, it becomes

\centers{$ \dec_{L_I}\big(\big[Y_j^I\big]\big) + \dec_{L_I}\big(\big[Y_{j-1}^I\big]\big) \, = \, a\big[{}^* \mathrm{R}_{L_I}^G (S_j)\big]_k + b \big[{}^* \mathrm{R}_{L_I}^G (S_{j+1})\big]_k$}

\noindent But by Lemma \ref{3lem1}, we know that the character of the restriction of $S_m$ is either zero or equal to $\dec_{L_I}\big(\big[Y_{m-1}^I\big]\big)$. Since the latter is non-zero for $m=j$, the previous equality forces ${}^* \mathrm{R}_{L_I}^G (S_{j})\neq \{0\}$ and $a=1$. In particular, $S_j$ is a simple $kG$-module which occurs with multiplicity one in any $\ell$-reduction of $Y_j$. The same argument applies to the equation

\centers{$ \dec_{L_I}\big(\big[{}^* \mathrm{R}_{L_I}^G (Y_{j-1})\big]\big)  \, = \, a'\big[{}^* \mathrm{R}_{L_I}^G (S_{j-1})\big]_k + b' \big[{}^* \mathrm{R}_{L_I}^G (S_{j})\big]_k.$}

\noindent It shows that $S_j$ occurs also as a composition factor with multiplicity $b'=1$ in any $\ell$-reduction of $Y_{j-1}$.\end{proof}

In particular, when $j$ is not equal to $m_{\zeta_j}$, the sum $\chi_j + \chi_{j-1}$ is always the character of an indecomposable projective $\Lambda G$-module. We will denote this module by $P_j$.

\begin{cons}From the two previous lemmas we deduce that the lines

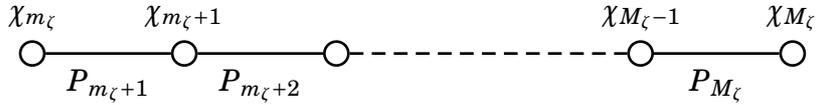
\centers{ \begin{pspicture}(10,1.4)
  \psset{linewidth=1pt}

  \cnode(0,0.6){5pt}{A}
  \cnode(2,0.6){5pt}{B}
  \cnode(4,0.6){5pt}{C}
  \cnode(8,0.6){5pt}{D}
  \cnode(10,0.6){5pt}{E}
  \ncline{A}{B}  \naput[npos=-0.1]{$\vphantom{\Big(}\chi_{m_\zeta}$} \naput[npos=1.1]{$\vphantom{\Big(}\chi_{m_\zeta+1}$}\nbput{$P_{m_\zeta+1}$}
  \ncline{B}{C} \nbput{$P_{m_\zeta+2}$}
  \ncline[linestyle=dashed]{C}{D}
  \ncline{D}{E}  \naput[npos=-0.1]{$\vphantom{\Big(}\chi_{M_\zeta-1}$} \naput[npos=1.1]{$\vphantom{\Big(}\chi_{M_\zeta}$}\nbput{$P_{M_\zeta}$}
\end{pspicture}}

\noindent are subtrees of the Brauer tree  $\Gamma$ of the principal  $\ell$-block of $G$. Moreover, the missing edges are labeled by the simple cuspidal $kG$-modules.
\end{cons}

\subsection{Cuspidal \texorpdfstring{$kG$}{kG}-modules\label{3se2}}

The most delicate step in the proof of the conjecture of Hiss-L\"ubeck-Malle consists in gluing the non-cuspidal branches to the exceptional node. If the conjecture holds, then by the previous work the missing edges are labeled by the simple cuspidal $kG$-modules, or equivalently by their projective cover. The character of the latter should therefore be given by $\chi_\mathrm{exc} + \chi_{\zeta_m}$.  We present here an explicit construction of these projective modules using the complex $\Rgc(\Y(\dot c),\Lambda)$ representing the cohomology of $\Y(\dot c)$. By the results in Section \ref{1se2} we know that this complex is perfect and thus provides a bunch of projective modules. Unfortunately, these are in general too big to be computed explicitly, and we need to factor them out according to the action of $F^\delta$. This can be achieved with the help of the tools introduced at the end of Section \ref{1se2}. However, to make this operation work perfectly, we will make the following assumption:

\centers{\begin{tabular}{cp{13cm}} \hskip-1mm $\mathbf{(W)}$ &  For all integer $j \in \intn{0}{h_0-1}$ such that $j = m_{\zeta_j}$, the generalized $(\lambda_j)$-eigenspace of $F^\delta$ on $b\Hc^\bullet(\Y(\dot c),\Lambda)$ is torsion-free. \end{tabular}}

\noindent By Theorem \ref{2thm4}, these eigenvalues correspond exactly to the eigenvalues of $F^\delta$ on the cohomology group in middle degree $\Hc^r(\X,K)$.

\sk

Under this assumption, the missing projective modules turn out to be homotopy equivalent to suitably chosen generalized eigenspaces
of $F^\delta$ on the complex $C = b\Rgc(\Y(\dot c),\Lambda)$. We shall rather work with a particular representative of this complex, with good finiteness properties that are needed to apply the constructions detailed in Section \ref{1se2}. 

\sk

By Corollary \ref{1cor1}, there exists a bounded complex $\mathscr{C}$ of finitely generated $(\Lambda G,$ $\Lambda \T^{cF})$-bimodules, projective as left and right modules, together with an equivariant map $f : C \longrightarrow \mathscr{C}$ and its inverse $g : \mathscr{C} \longrightarrow C$  in the category $K^b(\Lambda G \times (\T^{cF})^\mathrm{op})$. The Frobenius $F^\delta$ induces a morphism on $\mathscr{C}$ which we will denote by  $\mathscr{F} = f \circ F^{\delta} \circ g$.

\begin{prop}\label{3prop2}Let $j \in \intn{0}{h_0-1}$ such that $j = m_{\zeta_j}$ and $\lambda_j$ the associated eigenvalue of $F^\delta$. Assume that $\lambda_j$ satisfies the assumption $\mathrm{(W)}$. Then the complex $\mathscr{C}_{(\lambda_j)}$ is homotopy equivalent to an indecomposable projective module $P$ concentrated in degree $r$, and its character is given by

\centers{$ [P] \, = \chi_\mathrm{exc} + \chi_{m_\zeta}.$}

\end{prop}

\begin{proof} For the sake of simplicity we will write $\lambda = \lambda_j$ and $\zeta = \zeta_j$. From Proposition \ref{1prop1} we deduce that the complex $\mathscr{C}_{(\lambda)}$ has the following properties:

\begin{itemize}

\item $\mathscr{C}_{(\lambda)}$ is a bounded complex of finitely generated projective $\Lambda G$-modules;

\item the cohomology of the complex  $\mathscr{C}_{(\lambda)} \otimes_\Lambda k$ is concentrated in degrees $r, \ldots, 2r$ (it is already the case for $\mathscr{C}\otimes_\Lambda k$ since $\Y(\dot c)$ is an irreducible affine variety);

\item the cohomology groups of the complex $\mathscr{C}_{(\lambda)} \otimes_\lambda K$ vanish outside the degree $r$. This follows indeed from \cite[Corollary 9.9]{DeLu}, Theorem \ref{2thm4} and Fact \ref{2fact1}:

\centers{$\Hc^i(\mathscr{C}_{(\lambda)} \otimes_\Lambda K) \ \mathop{=}\limits^{\text{\ref{1prop1}}} \ b\Hc^i(\Y(\dot c),K)_{(\lambda)} \ \mathop{=}\limits^{\text{\cite{DeLu}}}  \ \Hc^i(\X,K)_{(\lambda)} \ \mathop{=}\limits^{\text{\ref{2fact1}}} \  \Hc^i(\X,K)_{\lambda} \ \mathop{=}\limits^{\text{\ref{2thm4}}}  \ 0.$}

\end{itemize}

\noindent Since the assumption $\mathrm{(W)}$ ensures that the groups $\Hc^i(\mathscr{C}_{(\lambda)} )$ are torsion-free, they are in fact zero for $i \neq r$. Therefore, by Lemma \ref{1lem1}, the complex $\mathscr{C}_{(\lambda)}$ can be represented up to homotopy by a projective module $P = \mathrm{H}^r(\mathscr{C}_{(\lambda)})$ concentrated in degree $r$. The character of this module corresponds (up to a sign) to the total character of the complex and is given by

\centers{$ [P] \, = \, \displaystyle \sum_{i=r}^{2r} (-1)^{i-r} \big[\mathscr{C}_{(\lambda)}^i\big] \, = \, \big[b\Hc^r(\Y(\dot c),K)_{(\lambda)}\big].$} 

\noindent Moreover, by Theorem \ref{2thm4}, this character has only one non-exceptional (i.e. unipotent) irreducible component, namely $\chi_{m_\zeta}$. Consequently, the only possible choice for the character of $P$ is $\chi_{\mathrm{exc}} + \chi_{m_\zeta}$. \end{proof}

In view of the results in the previous section, we have constructed the projective covers of the simple cuspidal $kG$-modules in the block. Following the previous notation, these simple modules will be denoted by $S_{m_{\zeta}}$ and their cover by $P_{m_{\zeta}}$. They label the edges of the Brauer tree connecting the non-cuspidal branches (or equivalently, the connected components of $\Gamma^\bullet$) to the exceptional node. This gives exactly the first part of the conjecture of Hiss, L\"ubeck and Malle:

\begin{cor}\label{3cor1}Under the assumption $\mathrm{(W)}$, Conjecture \hyperref[2conj1]{\emph{(HLM)}} holds.
\end{cor}

\subsection{Some numerical applications\label{3se3}}

We conclude this section by recording two direct consequences of the previous study, always under the assumption that $\mathrm{(W)}$ holds. 

\begin{prop} The simple unipotent cuspidal $KG$-modules in the principal $\ell$-block remain simple after $\ell$-reduction. 
\end{prop}

\begin{proof} By Theorem \ref{2thm4}, the simple unipotent cuspidal $KG$-modules in the block correspond to the modules $Y_{m_\zeta}$ such that  $m_{\zeta} = M_{\zeta}$. According to Conjecture \hyperref[2conj1]{{(HLM)}},  the associated nodes in the Brauer tree are extremities of the branches, so that any $\ell$-reduction of $Y_{m_\zeta}$ is isomorphic to the simple $kG$-module $S_{m_\zeta}$. 
\end{proof}

In Lusztig's classification, the irreducible unipotent characters fall into families \cite{Lu5}. By a case-by-case analysis, it has been checked that Lusztig's $a$-function (defined as the valuation of the degree of the unipotent character as a polynomial in the variable $q$) is constant on each family $\mathcal{F}$. This value will be denoted by $a(\mathcal{F})$. Similarly, the degree $A_\chi$ of the polynomial degree of any unipotent character $\chi$ depends only on the family.

\begin{prop} Let $\mathcalm{F}_1$, $\mathcal{F}_2$, \ldots, $\mathcal{F}_m$ be the $F$-stable families of unipotent characters lying in the principal $\ell$-block, ordered such that  $a(\mathcal{F}_1) \leq \cdots \leq a(\mathcal{F}_m)$.  Then the irreducible Brauer characters in the block can be labeled such that the decomposition matrix has a lower unitriangular shape:

\centers{$ D \, = \, \left( \begin{array}{cccc} I_{r_1} & 0 & \cdots & 0 \\ * & I_{r_2} & \ddots & \vdots \\
 \vdots & \ddots & \ddots & 0 \\ * & \cdots  & * & I_{r_m} \\ \end{array}\right)$}
 
 \noindent where each diagonal block is the identity matrix of the appropriate size.
\end{prop}

\begin{proof} By \cite[Lemma 5.11]{BMi2}, the power of $q^\delta$ occurring in the eigenvalue of $F^\delta$ associated to  $\chi$ is given by 

\centers{$ n_\chi \, = \, \displaystyle 2r - \frac{a_\chi + A_\chi}{h}\cdot $}

\noindent  Since these integers are constant on each family, we deduce from the shape of the tree that two distinct characters $\chi$ and $\eta$ lying the same family belong to distinct branches. In particular, $\dec_{kG}(\chi)$ and $\dec_{kG}(\eta)$ have no common irreducible component, which explains the identity matrices in the diagonal of $D$.

\sk

On the other side, one can check by a case-by-case analysis that in the principal $\ell$-block, $a_\chi \leq a_\eta$ if and only if $n_\eta \geq n_\chi$.  Since $n_\chi$ increases along each branch of the tree (see assertion $\mathrm{(ii)}$ in Conjecture \hyperref[2conj1]{(HLM)}), we deduce that the decomposition matrix has a lower triangular shape.\end{proof}

It is conjectured that these results (excluding the supercuspidality property) hold in general for any good prime number $\ell$ (see  \cite[Conjecture 3.4]{GeHis}). We will give in Section \ref{4se4} a more conceptual explanation of the previous proposition. 

\section{Towards Brou\'e's conjecture}

This last section aims at finalizing the study of the principal $\ell$-block of $G$. In all the results stated here, we assume that the following holds:

\centers{\begin{tabular}{cp{13cm}} \hskip-1mm $\mathbf{(S)}$ &  The $\Lambda$-modules $b\Hc^i(\Y(\dot c),\Lambda)$ are torsion-free. \end{tabular}}

\noindent Such an assumption is known to be valid for groups with $\mathbb{F}_q$-rank $1$ (since the corresponding Deligne-Lusztig variety is a irreducible affine curve) and for groups of type A$_n$  \cite{BR2}. Many other cases will be settled in a subsequent paper \cite{Du3} (see also \cite{Du2} for groups of type B$_n$, C$_n$ and ${}^2 $D$_n$). Under this assumption, the cohomology groups $b\Hc^i(\Y(\dot c),\Lambda)$ are integral versions of the groups $b\Hc^i(\Y(\dot c),K)$ and their $\ell$-reduction correspond to the groups $b\Hc^i(\Y(\dot c),k)$. Therefore, most of the information contained in the modular cohomology of $\Y(\dot c)$ is again obtained from  Lusztig's fundamental work.

\sk

We start by giving an explicit representative of the complex $b\Rgc(\Y(\dot c),\Lambda)$ in terms of the projective modules $P_j$ defined in the previous section. More precisely, we give, for all eigenvalue $\lambda$ of $F^\delta$, a  bounded complex of finitely generated projective $\Lambda G$-modules homotopy equivalent to the generalized $(\lambda)$-eigenspace of $F^\delta$. Under the assumption $\mathrm{(S)}$, we know explicitly the shape of the Brauer tree of the principal $\ell$-block. Surprisingly,  this complex turns out to be exactly the Rickard complex associated to the node labeled by $\lambda$ \cite[Section 4]{Ri}. From that observation we deduce that $b\Rgc(\Y(\dot c),\Lambda)$ induces a splendid equivalence between the principal $\ell$-blocks of $G$ and $\T^{cF} \rtimes C_W(c\sigma)$, as predicted in the geometric version of Brou\'e's conjecture \cite{BMa2}. We will conclude by giving two main consequences of this equivalence, namely the planar embedding of the Brauer tree and the unitriangularity shape of the decomposition matrix. As we have shown in section \ref{3se3}, the latter property is already a consequence of Conjecture \hyperref[2conj1]{(HLM)}. Nevertheless, we shall give here a conceptual approach to this result using the recent work of Chuang and Rouquier on perverse equivalences \cite{CR1}.

\subsection{Retrieving the complex\label{4se1}}

For the sake of notation, we shall simply denote by $\Y$ the Deligne-Lusztig variety associated to $\dot c$. We want to determine explicitly the contribution of each eigenspace of $F^\delta$ on the cohomology of $\Y$. This is in some sense a generalization of Proposition \ref{3prop2} which refers only to the "minimal" eigenvalues. We shall keep the same approach: using the torsion-free assumption one can find a small representative of  
the complex, in which only specific projective modules can show up. The determination of these modules is then achieved using the total character of the complex that we deduce from Lusztig's work. 

\sk

As in Section \ref{3se2}, we shall work with a specific representative of the cohomology complex $C = b\Rgc(\Y,\Lambda)$ with good finiteness properties: by Corollary \ref{1cor1}, there exists a bounded complex $\mathscr{C}$ of finitely generated $(\Lambda G,\Lambda \T^{cF})$-modules, whose terms are projective as both $\Lambda G$ and $\Lambda \T^{cF}$-modules such that $C$ is homotopy equivalent to $\mathscr{C}$. By transfer, the Frobenius $F^\delta$ induces an endomorphism $\mathscr{F}$ of $\mathscr{C}$ such that the images of $F^\delta$ and $\mathscr{F}$ coincide under the isomorphism $\mathrm{End}_{K^b(\Lambda G)}(\mathscr{C}) \simeq \mathrm{End}_{K^b(\Lambda G)}(C)$. 

\sk

Drawing inspiration from the case of $\mathrm{SL}_2(\mathbb{F}_q)$ detailled in \cite{Bon}, we express each relevant eigenspace of $\mathscr{F}$ on $\mathscr{C}$ in terms of the indecomposable projective $\Lambda G$-modules $P_j$:

\begin{thm}\label{4thm1}Let $j \in \intn{0}{h_0-1}$. The generalized $(\lambda_j)$-eigenspace of $\mathscr{F}$ on $\mathscr{C}$ is homotopy equivalent to the following complex, with non-zero terms in degrees $r,\ldots,r+j-m_{\zeta_j}$ only:

\centers{$  0 \longrightarrow P_{m_\zeta} \longrightarrow P_{m_\zeta+1} \longrightarrow \cdots \longrightarrow P_{j-1} \longrightarrow P_j \longrightarrow 0.$}

\noindent Moreover, the boundary maps $\overline{d}_i : \overline{P}_i \longrightarrow \overline{P}_{i+1}$ remain non-zero after $\ell$-reduction.
\end{thm}

Note that this complex is exactly the Rickard complex attached to the node labelled by $\chi_j$ in the Brauer tree \cite{Ri}. This observation will be fundamental in the next sections.

\sk

Before going into the details of the proof, let us recall some notation and basic properties of the Brauer tree $\Gamma$ of the principal $\ell$-block of $\Lambda G$.  The non-exceptional nodes are labelled by the unipotent characters $\chi_j$ lying in the block. There are as many simple $kG$-modules as unipotent characters, but we can distinguish
\begin{itemize}

\item the non-cuspidal modules $S_j$ with $j > m_{\zeta_j}$. Their projective cover $P_j$ has character $\chi_j +\chi_{j-1}$, and hence labels the edge connecting the nodes associated to $\chi_j$ and $\chi_{j-1}$;  

\item the cuspidal modules $S_{m_\zeta}$. The character of the corresponding projective cover is given by $[P_{m_\zeta}] = \chi_{m_\zeta} + \chi_{\mathrm{exc}}$. 

\end{itemize}

\noindent Moreover, from the particular shape of the tree (given by the conjecture of Hiss-L\"ubeck-Malle) one can deduce that for $j > m_{\zeta_j}$, the $\ell$-reduction of the projective module $P_j$ is given by

\centers{$ \overline{P}_j \, = \, %\left(
\hskip-1.3mm \begin{array}{ccc} S_j \\[3pt] S_{j-1} \oplus S_{j+1} \\[3pt] S_j \end{array} \hskip-1.3mm %\right)
$}

\noindent if $j < M_{\zeta_j}$, that is in the case where the node labelled by $\chi_j$ is not at an extremity of the tree. Otherwise it is given by

\centers{$ \overline{P}_{M_\zeta} \, = \, %\left(
\hskip-1.3mm \begin{array}{ccc} S_{M_\zeta} \\[3pt] S_{M_{\zeta}-1} \\[3pt] S_{M_\zeta} \end{array} \hskip-1.3mm %\right)
$}

\noindent On the other hand, if the planar embedding of $\Gamma$ is not specified, one cannot know precisely what the modules $\overline{P}_{m_\zeta}$ will look like. We know, however, that they have simple head and simple socle, both isomorphic to $S_{m_\zeta}$, and that their class in the Grothendieck group $K_0(kG$-$\mathrm{mod})$ is given by

\centers{$ \big[\overline{P}_{m_\zeta}\big]_k \, = \,  [S_{m_\zeta}]_k + [S_{m_\zeta+1}]_k + \displaystyle \frac{|\T^{cF}|_\ell-1}{h_0}\sum_{\xi} \,  [S_{m_\xi}]_k $.}

\noindent The number $(|\T^{cF}|_\ell-1)/{h_0}$ corresponds actually to the multiplicity of the exceptional node, that is the number of irreducible components of $\chi_{\mathrm{exc}}$.

\begin{de} We define the \emph{height in the tree $\Gamma$} of an indecomposable projective $\Lambda G$-module  $P$ lying in the principal block to be the minimal length of a path from the exceptional node to the edge labelled by $P$. It will be denoted by $\hg(P)$.  
\end{de}

We shall adopt the convention $\hg(P_{m_\zeta})=0$. By extension, the height of any finitely generated projective module lying in the block will be the maximal height of its indecomposable factors. Also, the height of a simple $kG$-module in the block will be naturally defined as the height of its projective cover, so that $\hg(P_j) = \hg(S_j) = j-m_{\zeta_j}$.

\begin{rmk}\label{4rmk1}If $P$ is a projective module of height $n$, then the height of any of its composition factors is at most $n+1$. 
\end{rmk}

In order to prove Theorem \ref{4thm1}, we start by scanning the complex from the left to the right by removing the highest indecomposable factors. We shall assume that $j$ is different from $m_{\zeta_j}$ since this case has been treated in Proposition \ref{3prop2}. We first obtain the following representative:

\begin{lem}\label{4lem1}Let $\zeta = \zeta_j$. The complex $\mathscr{C}_{(\lambda_j)}$ is homotopy equivalent to a bounded complex of finitely generated projective $\Lambda G$-modules 

\centers{$ 0 \longrightarrow R_{m_{\zeta}} \mathop{\longrightarrow}\limits^{\delta_{m_{\zeta}}} R_{m_{\zeta} +1} \longrightarrow \cdots \longrightarrow R_{j-1} \mathop{\longrightarrow}\limits^{\delta_{j-1}} R_j \longrightarrow 0$}

\noindent satisfying $\hg(R_i) \leq \hg(P_i)$ for all $i = m_{\zeta_j}, \ldots,j$.

\end{lem}

\begin{proof} The assumption $\mathrm{(S)}$ together with the results \cite[Corollary 9.9]{DeLu} and \ref{2thm4} ensure that the cohomology of the complex $\mathscr{C}_{(\lambda_j)}$ vanishes outside the degrees $r$ and $r + j-m_{\zeta}$. By Lemma \ref{1lem1} we deduce that the latter can be represented by a bounded complex of finitely generated $\Lambda G$-modules

\centers{$ 0 \longrightarrow Q_{m_{\zeta}} \mathop{\longrightarrow}\limits^{d_{m_{\zeta}}} Q_{m_{\zeta} +1} \longrightarrow \cdots \longrightarrow Q_{j-1} \mathop{\longrightarrow}\limits^{d_{j-1}} Q_j \longrightarrow 0$}

\noindent satisfying the following properties:

\begin{itemize}

\item $\mathrm{Im}\, d_i \, = \, \mathrm{Ker}\, d_{i+1}$ for all  $i = m_{\zeta}, \ldots, r+j-m_{\zeta}-1$;

\item  $\mathrm{H}^r(\mathscr{C}_{(\lambda_j)}) = \mathrm{Ker}\, d_{m_{\zeta}}$ is a $\Lambda G$-lattice with character $n_j \chi_{\mathrm{exc}}$ for some non-negative integer $n_j$;

\item $\mathrm{H}^{r+j-m_{\zeta}}(\mathscr{C}_{(\lambda_j)}) = Q_j/\mathrm{Im} \, d_{j-1}$ is a $\Lambda G$-lattice with character $\chi_j$.

\end{itemize}

\noindent The integer $n_j$ is actually non-zero otherwise $\chi_j$ would be a a linear combination of projective characters.  Moreover, we know from Section \ref{2se2} that the multiplicity of $\chi_{\mathrm{exc}}$ in the cohomology of $\Y$ is $\sum_i n_i = |C_W(c\sigma)| = h_0$, which forces each integer $n_i$ to be equal to $1$.

\begin{rmk}\label{4rmk2}This proves incidentally that the $\ell$-reduction of any eigenvalue of $F^\delta$ on the cohomology of $\Y$ (and not only $\X$) is an $h_0$-th root of unity.
\end{rmk}

Let us prove by induction on $n$ that  $\mathscr{C}_{(\lambda_j)}$ is homotopy equivalent to a complex of the following form

\centers{$ 0 \longrightarrow R_{m_\zeta} \mathop{\longrightarrow}\limits^{\delta_{m_\zeta}} R_{m_\zeta +1}   \mathop{\longrightarrow}\limits^{\delta_{m_\zeta+1}}
 \cdots  \mathop{\longrightarrow}\limits^{\delta_{n-2}\vphantom{\delta_{m_\zeta}}}R_{n-1} 
  \mathop{\longrightarrow}\limits^{d_{n-1}'\vphantom{\delta_{m_\zeta}}} Q_n'  \mathop{\longrightarrow}\limits^{d_n'\vphantom{\delta_{m_\zeta}}} \cdots \mathop{\longrightarrow}\limits^{d_{j-1}'\vphantom{\delta_{m_\zeta}}} Q_j' \longrightarrow 0$}

\noindent with the $R_i$'s satisfying $\hg(R_i) \leq \hg(P_i)$ for all $i<n$. The case $n = m_\zeta$ is obtained from the previous analysis. 

\sk

Assume then that $\mathscr{C}_{(\lambda_j)}$ is homotopy equivalent to the previous complex for some integer $n \geq m_\zeta$ and let us try to symplify $Q_n'$. For the sake of notation, we shall write 
$d : A \longrightarrow B$ instead of $d_n' : Q_n' \longrightarrow Q_{n+1}'$. Let $P_m$ be any indecomposable direct summand  of $A$ that is assumed to be strictly higher than $P_n$ (and hence of non-zero height). If we decompose $A$ into $A = P_m \oplus A'$ then one can check that the following properties hold:

\begin{itemize}

\item $d$ restricts to an injective map from $P_m$ to $B$: by construction the character of $P_m$ is $\chi_m + \chi_{m-1}$ whereas the character of $\mathrm{Ker}\, d$ equals  

\centers{$ [\mathrm{Ker}\, d] \, = \, [R_{n-1}]-[R_{n-2}]+ \cdots + (-1)^{n-m_\zeta+1}[R_{m_\zeta}]+(-1)^{n-m_\zeta} \chi_{\mathrm{exc}}.$}

\noindent Now, by assumption, neither $\chi_m$ nor $\chi_{m-1}$ can occur in this expression. Therefore, the module $P_m \cap \mathrm{Ker}\, d$ has zero character; since it is torsion-free, it must be the zero module.

\item the quotient module $B/d(P_m)$ is torsion-free: we can use the following exact sequence  of $kG$-modules:

\centers{$ 0 \mathop{\longrightarrow}\limits \mathrm{Tor}_1^\Lambda\big(B/d(P_m),k\big) \mathop{\longrightarrow}\limits \overline{P}_m \mathop{\longrightarrow}\limits^{\overline{d}} \overline{B}. $}

\noindent Consequently, it is sufficient to show that $\mathrm{Ker}\, \overline{d} \cap \overline{P}_m$ is zero. Let us consider the class of $\mathrm{Ker}\, \overline{d}$ in $K_0(kG$-$\mathrm{mod})$, which is given by

\centers{$ \big[\mathrm{Ker}\, \overline{d}\big] \, = \, \big[\overline{R}_{n-1}\big]-\big[\overline{R}_{n-2}\big]+ \cdots + (-1)^{n-m_\zeta+1}\big[\overline{R}_{m_\zeta}\big]+(-1)^{n-m_\zeta} \big[\mathrm{Ker}\, \overline{\delta}_{m_\zeta}\big].$}

\noindent Using the assumption $\mathrm{(S)}$ and the universal coefficient formula, we can identify the 
 $kG$-module $\mathrm{H}^r(\mathscr{C}_{(\lambda_j)} \otimes_\Lambda k) = \mathrm{Ker}\, \overline{\delta}_{m_\zeta}$ with the $\ell$-reduction of $\mathrm{H}^r(\mathscr{C}_{(\lambda_j)})$ and can thus write

 \centers{$ \mathrm{Ker}\, \overline{\delta}_{m_\zeta} \, \simeq \, \mathrm{H}^r(\mathscr{C}_{(\lambda_j)})\otimes_\Lambda k \, = \, \overline{\mathrm{Ker} \, \delta_{m_\zeta}}$.}
 
 \noindent Since  $\mathrm{H}^r(\mathscr{C}_{(\lambda_j)})$ is cuspidal (it is torsion-free and its character is $\chi_{\mathrm{exc}}$), the module  $\mathrm{Ker}\, \overline{\delta}_{m_\zeta}$ has only cuspidal composition factors.  Therefore, the simple module $S_m$, which by definition has the same height as $P_m$, can occur as a composition factor
  
 \begin{itemize}
 
 \item neither in $\overline{R}_i$ for $i<n$, for the height of any of its irreducible components is at most $\hg(P_n)$ by assumption (see Remark \ref{4rmk1}); 
 
  \item nor in $\mathrm{Ker}\, \overline{\delta}_{m_\zeta}$ whose irreducible components are cuspidal and have hence height zero.
 
 \end{itemize}

\noindent From the expression of $[\mathrm{Ker}\, d]$ given previously we deduce that $S_m$ is not a composition factor of $\mathrm{Ker}\, d$. Since it is the only simple module of the socle of $\overline{P}_m$, this forces $ \mathrm{Ker}\, \overline{d} \cap \overline{P}_m$ to be zero.
 
\end{itemize}

\noindent Consequently, we can decompose $B$ into $B = d(P_m) \oplus B'$ so that $d$ induces an isomorphism between $P_m$ and $d(P_m)$. If we define $d' : A' \longrightarrow B'$ to be the composition of the restriction $d_{|A'}$  with the projection  $B \longrightarrow B'$, then one can construct the following complex

\centers{$ 0 \longrightarrow R_{m_\zeta} \mathop{\longrightarrow}\limits^{\delta_{m_\zeta}}
 \cdots  \mathop{\longrightarrow}\limits R_{n-1} 
  \mathop{\longrightarrow}\limits A'  \mathop{\longrightarrow}\limits^{d'\vphantom{\delta_{m_\zeta}}} B' \longrightarrow \cdots \mathop{\longrightarrow}\limits^{d_{j-1}' \vphantom{\delta_{m_\zeta}}} Q_j' \longrightarrow 0$}

\noindent which is clearly homotopy equivalent to $\mathscr{C}_{(\lambda_j)}$. By removing repeatedly all the indecomposable direct summands of $A$ that are higher to $P_n$ we obtain the projective module $R_{n}$. \end{proof}

In the previous lemma, we have modified the complex from the left to the right by removing the superfluous projective modules. We now use the same method in the other direction to finish the proof of Theorem \ref{4thm1}.

\begin{proof}[Proof of the theorem] We argue once again by induction: we show that up to homotopy $\mathscr{C}_{(\lambda_j)}$ can be written as

\centers{$0 \longrightarrow R_{m_\zeta} \mathop{\longrightarrow}\limits^{\delta_{m_\zeta}} R_{m_\zeta +1}  \longrightarrow \cdots \mathop{\longrightarrow}\limits^{\delta_{n-1}\vphantom{\delta_{m_\zeta}}} R_{n} \mathop{\longrightarrow}\limits^{\delta \vphantom{\delta_{m_\zeta}}} P_{n+1} \longrightarrow \cdots \longrightarrow P_j \longrightarrow 0$}

\noindent where the modules $R_i$'s satisfy the condition $\hg(R_i) \leq \hg(P_i)$. Note that the case $n=j$ has been treated in the previous lemma.

\sk

Assume that we are working with the previous complex for some integer $n \leq j$. In that case, the character of the $\Lambda G$-module $\mathrm{Coker}\, \delta_{n-1} = R_n/ \mathrm{Im}\, \delta_{n-1}$ is given by

\centers{$\begin{array}{r@{\ \, = \, \ }l} [\mathrm{Coker}\, \delta_{n-1} ] &  [P_{n+1}]-[P_{n+2}]+ \cdots + (-1)^{j-n+1}[P_j]+ (-1)^{j-n} \big[\mathrm{H}^{r+j-m_\zeta}(\mathscr{C}_{(\lambda_j)})\big] \\[4pt]
& (\chi_n + \chi_{n+1}) - (\chi_{n+1} + \chi_{n+2}) + \cdots + (-1)^{j-n}\chi_j \\[4pt]
[\mathrm{Coker}\, \delta_{n-1}] & \chi_n + (\chi_{n+1} - \chi_{n+1}) - (\chi_{n+2} -\chi_{n+2}) +  \cdots \, = \, 
\chi_n. \end{array}$}

\noindent In addition, it is torsion-free: it is indeed isomorphic either to $\mathrm{Im} \, \delta$ when $n<j$ or to $\mathrm{H}^{r+j-m_\zeta}(\mathscr{C}_{(\lambda_j)})$ when $n=j$. Therefore, the head of  the $\Lambda G$-module $\mathrm{Coker}\, \delta_{n-1}$ consists of at most two simple modules, namely $S_n$ and $S_{n+1}$. Let $P$ be a projective cover of $\mathrm{Coker}\, \delta_{n-1}$. The canonical projection  $R_n \twoheadrightarrow \mathrm{Coker}\, \delta_{n-1}$ factors through $P$ so that $\hg(P) \leq \hg(R_n) \leq \hg(P_n)$. Consequently, $S_{n+1}$ cannot be in the head of $P$ which forces $P$ to be exactly $P_n$. This allows us to decompose the module $R_n$ into $R_n = P_n \oplus R'$ with $R' \subset \mathrm{Im}\, \delta_{n-1}$.

\sk         

For the sake of notation we shall now write $\partial : A \longrightarrow B$ instead of $\delta_{n-1} : R_{n-1} \longrightarrow R_n$. We can argue as in \cite{Bon}: since $R'$ is a projective module, the map $\partial^{-1}(R') \twoheadrightarrow R'$ splits. The image of the corresponding splitting map is a sub-module $R''$ of $A$ isomorphic to $R'$, such that the quotient $A / R''$ is torsion-free. Indeed, one can embed $(A/R'')/(\partial^{-1}(R')/R'')$ in $B/R'$ via $\partial$ and both $B/R' \simeq P_n$ and $\partial^{-1}(R')/R'' \simeq \partial^{-1}(R') \cap \mathrm{Ker}\, \partial \subset A$ are torsion-free. Since $R''$ is projective and $A/R''$ is torsion-free, then the map $R'' \hookrightarrow A$ must be a retraction (see the proof of Lemma \ref{1lem1} for more details), and $A$ decomposes into $A = R'' \oplus A'$ as a $\Lambda G$-module. It follows that $\mathscr{C}_{(\lambda_j)}$ is homotopy equivalent to

\centers{$0 \longrightarrow R_{m_\zeta} \longrightarrow \cdots \longrightarrow A' \mathop{\longrightarrow}  P_n \mathop{\longrightarrow} P_{n+1} \longrightarrow \cdots \longrightarrow P_j \longrightarrow 0$}

\noindent where the heights of the modules $R_{m_\zeta},\ldots, R_{n-1},A'$ satisfy the conditions given in Lemma \ref{4lem1}. At the last step of the induction we have  obtained  a complex of the following form:

\centers{$0 \longrightarrow R \longrightarrow P_{m_\zeta+1} \longrightarrow \cdots \longrightarrow P_j\longrightarrow 0.$}

\noindent Finally, the character of $R$ (and therefore $R$ itself) can be deduced from the total character of the complex, which is here equal  to $\chi_\mathrm{exc} + (-1)^{j-m_\zeta} \chi_j$.\end{proof}

\begin{rmk} One can actually make the boundary maps $d_i  : P_i \longrightarrow P_{i+1}$ explicit. The characters of the kernel and the image of $d_i$ can be deduced from the cohomology of the variety $\Y$ with coefficients in $K$. One gets respectively $[\mathrm{Ker}\,  d_i ] = \chi_{i-1}$ and $[\mathrm{Im}\, d_i] = \chi_i$. Besides, we have shown in the course of the proof that the cokernel of $d_i$ is a torsion-free module with character $[\mathrm{Coker}\, d_i] = \chi_{i+1}$. Consequently, the $\ell$-reduction of $d_i$ factors through

\centers{$\begin{psmatrix}  \begin{array}{ccc} \textcolor{violet}{S_i} \\[3pt] S_{i-1} \oplus \textcolor{violet}{S_{i+1}} \\[3pt] S_i \end{array} & \begin{array}{c} \textcolor{violet}{S_i} \\[3pt] \textcolor{violet}{S_{i+1} }\end{array} & \begin{array}{ccc} S_{i+1} \\[3pt] \textcolor{violet}{S_{i}} \oplus S_{i+2} \\[3pt] \textcolor{violet}{S_{i+1}} \end{array}
\psset{arrows=H->,hooklength=1.5mm,hookwidth=-1.2mm,nodesep=5pt} 
\everypsbox{\scriptstyle} 
\ncline{1,2}{1,3}    
\psset{arrows=->>,nodesep=5pt} 
\ncline{1,1}{1,2}
\end{psmatrix}$}

\noindent and the complex $\overline{C} = b\Rgc(\Y,k)$ is, as expected, homotopy equivalent to the Rickard complex associated to the Brauer tree $\Gamma$ \cite{Ri}. 

\end{rmk}

As a byproduct, we obtain many properties of the cohomology of the Deligne-Lusztig variety $\Y$. We shall use the followings: 

\begin{cor}\label{4cor1}The complex $C = b\Rgc(\Y,\Lambda)$ is a tilting complex for $\Lambda G b$. In other words, it is a perfect complex of $\Lambda G   b$-modules satisfying the following properties:

 \begin{itemize}
 \item[$\bullet$] $\mathrm{Hom}_{K^{b}(\Lambda G{b})}({C},{C}[i]) = 0$ for $i\neq 0$ ;  

 \item[$\bullet$] $\mathrm{add}\, C$ generates $K^b(\Lambda G{b}$-$\mathrm{mod})$ as a triangulated category.  
 \end{itemize}

\noindent Moreover, the endomorphism algebra $\mathrm{End}_{K^{b}(\Lambda G b)}(C)$ is free over  $\Lambda$ and is homotopy equivalent to  $\mathrm{Hom}_{\Lambda G b}^\bullet(C,C)$ as a complex of $(\Lambda \T^{cF},\Lambda \T^{cF})$-bimodules.
\end{cor}

\begin{proof}  $\overline{C}$ is homotopy equivalent to a Rickard complex and, as such, it is a tilting complex \cite[Theorem 4.2]{Ri}. The same holds for $C$ since it is the unique tilting complex lifting $\overline{C}$ (see \cite[Proposition 3.1 and Theorem 3.3]{Ri3}). Consequently,  the cohomology of both $E = \mathrm{Hom}_{\Lambda G}^\bullet(C,C)$ and $\overline E = E \otimes_\Lambda k$ is zero outside the degree $0$.  Now $C$ is a perfect complex, and therefore $\mathrm{Hom}_{\Lambda G}^\bullet(C,C)$ is homotopy equivalent to a complex of finitely generated projective $(\Lambda \T^{cF},\Lambda \T^{cF})$-bimodules. By Lemma \ref{1lem1}, we deduce that $E$ is homotopy equivalent to $\mathrm{H}^0(E)\simeq   \mathrm{End}_{K^{b}(\Lambda G)}(C)$. In particular, the latter module is free over $\Lambda$.
\end{proof}

\begin{cor}\label{4cor2}The natural homomorphism of $\Lambda$-algebra

\centers{$\mathrm{End}_{K^{b}(\Lambda G b)}\big(b\Rgc(\Y,\Lambda)\big) \longrightarrow \mathrm{End}_{\Lambda G b}^\mathrm{gr}\big(\mathrm{H}_c^\bullet(\Y,\Lambda)\big)$}

\noindent is injective.
\end{cor}

\begin{proof} Since the category $KG$-$\mathrm{mod}$ is semi-simple, we have the following commutative diagram, with $C = b\Rgc(\Y,\Lambda)$

\centers{$ \begin{psmatrix}   \mathrm{End}_{D^{b}(\Lambda G b)}(C) &   \mathrm{End}_{D^{b}(K Gb)}(C)
\\   \mathrm{End}_{\Lambda G b}^{\mathrm{gr}} \big( \mathrm{H}^\bullet(C)\big) & \mathrm{End}_{K G b }^{\mathrm{gr}} \big( \mathrm{H}^\bullet(C)\big) 
\psset{arrows=->,nodesep=5pt} 
\everypsbox{\scriptstyle} 
\ncline{1,1}{1,2}^{\iota}    
\ncline{1,1}{2,1}
\ncline{<->}{1,2}{2,2}%>{\rotatebox{90}{\sim}} 
\ncline{2,1}{2,2}
\end{psmatrix}$}

\noindent By the previous corollary, the $\Lambda$-module $ \mathrm{End}_{D^{b}(\Lambda G b )}(C) \simeq  \mathrm{End}_{K^{b}(\Lambda G b)}(C)$ is free over $\Lambda$. We deduce that the map $\iota$, as well as the first vertical arrow, are into. \end{proof}

\begin{cor}\label{4cor3}The image of $F^h-1$ in $\mathrm{End}_{K^{b}(k G \overline{b})}\big(\overline{b}\Rgc(\Y,k)\big)$ is nilpotent.
\end{cor}

\begin{proof} Let $A =  \mathrm{End}_{K^{b}(\Lambda G b)}\big(b\Rgc(\Y,\Lambda)\big)$. By Corollary \ref{4cor1}, the $\ell$-reduction of $A$ is exactly $\mathrm{End}_{K^{b}(k G \overline{b})}\big(\overline{b}\Rgc(\Y,k)\big)$. Let us denote by $\chi$ the minimal polynomial of $F^\delta$ on $\Hc^\bullet(\Y,K)$. By the previous corollary and the assumption $\mathrm{(S)}$,  the image of $\chi(F^\delta)$ in $A$ is zero. But the eigenvalues of $F^{\delta}$ reduce to $h_0$-th roots of unity modulo $\ell$ (see Remark \ref{4rmk2}) and hence the class of $F^h-1$ in $\overline{A}$ is a nilpotent element (recall that $h =h_0 \delta$).\end{proof}

\subsection{Brou\'e's conjecture\label{4se2}}

The original version of \emph{Brou\'e's abelian defect group conjecture} \cite{Bro1} predicts that the module categories of a block and its Brauer correspondent are derived equivalent, provided that the defect of the block is an abelian group. More precisely, given a block with abelian defect group $H$, represented by an idempotent $b$, and $c = \mathrm{Br}(b)$ the corresponding block of $N_G(H)$, there exists an equivalence

\centers{$ D^b(\Lambda G b$-$\mathrm{Mod}) \, \mathop{\longrightarrow}\limits^\sim  D^b(\Lambda N_G(H) c$-$\mathrm{Mod}).$}

\noindent Such an equivalence induces a \emph{perfect isometry} between the Grothendieck groups carrying numerous arithmetical information. There are indeed many numerical consequences we can deduce from it, e.g. it preserves the number of irreducible characters (ordinary and modular) as well as the similarity invariants of the Cartan matrix \cite{Bro2}. Up to now, the version of this conjecture is known to hold in the following cases:

\begin{itemize}

\item with restrictions on the defect group: if $H$ is a cyclic group \cite{Ri}, \cite{Lin1} and \cite{Rou1} or isomorphic to the Klein group $\mathbb{Z}/2 \mathbb{Z} \times \mathbb{Z}/2 \mathbb{Z}$  \cite{Ri2} and \cite{Lin2};

\item with restrictions on $G$: for $\ell$-solvable groups \cite{Da}, \cite{Pu1} and \cite{HaLi}, symmetric groups and general linear groups \cite{CR} or for finite reductive groups when $\ell$ divides $q-1$ \cite{Pu2}.

\end{itemize}

\noindent Many other particular cases have been handled, and a lot of evidences for this conjecture to hold have been collected.

\sk

It is unclear whether there should exist a canonical way to construct this equivalence. However, when $G$ is a finite reductive group, it is expected to be induced by the cohomology of certain Deligne-Lusztig varieties. This is known as the \emph{geometric version of Brou\'e's conjecture}, as  stated in \cite{BMa2} and \cite{BMi2}.  Note that varieties associated to Levi subgroups $-$ and not only tori $-$ can be involved in this description. However, if the order of $q$  modulo $\ell$ is assumed to be a regular number $d$, then it is sufficient to consider Deligne-Lusztig varieties $\Y(\dot  w)$ associated to elements that satisfy the following properties:

\begin{itemize}

\item[$\mathbf{(B1)}$] $w\sigma$ is a good \emph{$d$-regular element} \cite{BMi2}.
\item[$\mathbf{(B2)}$] $\ell$ divides $|T_w|$ but does not divide $[G: T_w]$.
\end{itemize}

\noindent We will also assume that the prime number $\ell$ is large:

\begin{itemize}

\item[$\mathbf{(B3)}$] $\ell$ does not divide $|W^F|$

\end{itemize} 

\noindent in order to ensure that $N_G(\T_w) = N_G(T_w)$.

\sk

In this geometric framework, the defect group $H$ of the principal block (which is an Sylow $\ell$-subgroup of $G$) can be chosen to be a subgroup of $T_w$, and the property of $w\sigma$ to be regular forces $C_\G(H)$ to be the exactly the torus $ \T_w$, leading to $N_G(\T_w)=N_G(H)$. Then the geometric version of Brou\'e's conjecture predicts that the perfect complex $b\Rgc(\Y(\dot w),\Lambda)$ induces a splendid Rickard equivalence (and in particular a derived equivalence) between the principal $\ell$-blocks of  $G$ and $N_G(T_w)$. More precisely,

\begin{conj}[Brou\'e]\label{4conj1}Under the previous assumptions, there exists a bounded complex $D$ of $\big(\Lambda G b, \Lambda N_G(T_w)\big)$-bimodules such that

\begin{enumerate} 
\item[$\mathrm{(i)}$] The restrictions of $b\Rgc(\Y(\dot w),\Lambda)$ and $D$ to the category of bounded complex of $(\Lambda G b,  \Lambda T_w)$-bimodules are homotopy equivalent.
\item[$\mathrm{(ii)}$] The complex $D$ induces a splendid Rickard equivalence between the  principal $\ell$-blocks of  $G$ and $N_G(T_w)$.

\end{enumerate}
\end{conj}

\begin{rmk} In the case where $w$ is not assumed to be good, the disjunction property of the cohomology 
\begin{equation}\label{disjeq} \forall\,  i \neq 0 \quad \mathrm{Hom}_{D^b(\Lambda Gb)}\big(b\Rgc(\Y(\dot w),\Lambda),b\Rgc(\Y(\dot w),\Lambda)[-i]\big) = 0
\end{equation}
\noindent does not always hold. When $(\G,F)$ has no twisted components of type ${}^2$B$_2$, ${}^2$F$_4$ ou ${}^2$G$_2$, the other assumptions on $w$ tell us  that $(\T_w,1)$ is a $d$-cuspidal pair, and \cite[Theorem 5.24]{BMM} turns out to be nothing but the numerical reflect of this conjecture.
\end{rmk}

Among the elements $w$ satisfying the previous assumptions, Coxeter elements of $(W,F)$ play a particular role, for they have the following remarkable property:

\begin{itemize}

\item[$\mathbf{(B4)}$]  $C_W(w\sigma)$ is a cyclic group generated by $wF(w)\cdots F^{\delta -1}(w)$.

\end{itemize}

\noindent In this case, the action of the Frobenius $F^\delta$ provides a natural way to extend the action of $T_w \simeq \T^{wF}$ on the complex $b\Rgc(\Y(\dot w),\Lambda)$ to an action of the normalizer $N_G(T_w)$. Using this crucial observation, Rouquier has reduced Brou\'e's conjecture to the disjunction property \ref{disjeq} for the cohomology with coefficients in $k$ \cite{Rou}. We know from the previous section that this  property is satisfied whenever $w$ is a Coxeter element, so that Conjecture \ref{4conj1} can now be directly deduced from Corollary \ref{4cor1} and \cite[Theorem 4.5]{Rou}:

\begin{thm}\label{4thm3}Under the assumption $\mathrm{(S)}$, the geometric version of Brou\'e's conjecture  holds for Coxeter elements.
\end{thm}

Nevertheless, it is worth giving details of the proof of \cite[Theorem 4.5]{Rou} since it helps to understand how the actions of $v = cF(c) \cdots F^{\delta -1}(c)$ and $F^{\delta}$  on the cohomology of $\Y$ are related.

\sk

Let us consider the algebra $A = \mathrm{End}_{K^b(\Lambda G b )}(C)$ associated to the complex $C = b \Rgc(\Y,\Lambda)$. By the disjunction property, $A$ is homotopy equivalent to $\mathrm{End}_{\Lambda G b}^\bullet(C)$ as a complex of $(\Lambda \T^{cF}, \Lambda \T^{cF})$-bimodules (see Corollary \ref{4cor1}). Moreover, it is free over $\Lambda$ and thus satisfies 

\centers{$ \overline{A} \, = \, A \otimes_\Lambda k \, \simeq \,  \mathrm{End}_{K^b(k G \overline{b})}(\overline{C})$.}

\noindent Since the action of $F^\delta$ on $C$ commutes with the action of $G$, we get a canonical morphism $\phi : \Lambda \T^{cF} \rtimes \langle F^\delta \rangle_{\mathrm{mon}} \longrightarrow A$  which turns out to be surjective (for more details see the proof of \cite[Theorem 4.5]{Rou}). We will denote by $\tau$ the image of $F^h$ and by  $\langle \tau \rangle$ the subalgebra of $A$ that it generates. By Corollary \ref{4cor3}, $\overline \tau -1 $ is a nilpotent element of $\overline A$. Therefore,  we can apply Hensel's lemma to the ideal $\mathfrak{m} = \mathrm{Nil}\langle \tau \rangle+ \ell \langle \tau \rangle$ in order to obtain a element $\alpha \in \langle \tau \rangle$ such that $\alpha^{h_0} = \tau$ and $\overline{\alpha} - 1 $ is nilpotent (recall that $h$ is prime to $\ell$). 
Then we can deform  $\phi$ into a homomorphism of algebras

\centers{$\psi \, : \, \Lambda \T^{cF} \rtimes C_W(c\sigma) \longrightarrow A$}

\noindent by setting $\psi(v^{-1}) = \alpha^{-1} \phi(F^{\delta})$. This defines an action of $v$ on $C$ in the homotopy category such that the image of $v^{-1}- F^\delta$ in $\overline{A}$ is a nilpotent element. To conclude, we use \cite[Lemma 4.9]{Rou} to construct a complex $D$ of 
$(\Lambda G b,\Lambda \T^{cF} \rtimes C_W(c\sigma))$-bimodules and a homotopy equivalence $f$ between the restrictions of $C$ and $D$ to the category of complexes of $(\Lambda G b ,\Lambda \T^{cF} )$-bimodules such that the following diagram is commutative:

\centers{$\begin{psmatrix} \Lambda \T^{cF} \rtimes C_W(c\sigma) & \mathrm{End}_{K^b(\Lambda G b)} (C) \\ &   \mathrm{End}_{K^b(\Lambda G b)} (D) 
\psset{arrows=->,nodesep=5pt} 
\everypsbox{\scriptstyle} 
\ncline{1,1}{1,2}^{\psi}    
\ncline{1,1}{2,2}_{\mathrm{can}} 
\ncline{<->}{1,2}{2,2}>{f}  
\end{psmatrix}$}

\noindent We deduce that the functor $D^\vee \otimes_{\Lambda G} - $ induces the expected splendid Rickard  equivalence between the principal $\ell$-blocks of $G$ and $\T^{cF} \rtimes C_W(c\sigma)$.

\begin{rmk} The projection $N_G(T_c) \twoheadrightarrow C_W(c\sigma)$ given by Proposition \ref{2prop1} does not split in general, and the groups $N_G(T_c)$ and $\T^{cF} \rtimes C_W(c\sigma)$ are not isomorphic. However, it is proven in \cite{Rou} that the algebras $\Lambda N_G(T_c) $ and $\Lambda \T^{cF} \rtimes C_W(c\sigma)$ become isomorphic as soon as $\ell$ satisfies the assumptions given at the beginning of Section \ref{2se2}.
\end{rmk}

In the remaining sections, we shall investigate further properties of the functor $D^\vee \otimes_{\Lambda G} -$ using explicit representatives coming from Section \ref{4se1}.

\subsection{Planar embedding of the Brauer tree\label{4se3}}

We start by constructing a representative $\mathscr{D}$ of $D$ with good finiteness properties, as we did for $C = b\Rgc(\Y,\Lambda)$. The restriction of $D$ to the category of complexes of $(\Lambda G, \Lambda \T^{cF})$-bimodules is homotopy equivalent to $C$, and hence to $\mathscr{C}$ which is a bounded complex of projective modules. 

\sk

It remains to define the action of $C_W(c\sigma)$ on this complex using the action of $v$ on $D$. By transfer, there exists an endomorphism $\widetilde v$ of $\mathscr{C}$ such that $v$ and $\widetilde v$ coincide under the isomorphism $\mathrm{End}_{K^b(\Lambda G)} (D) \simeq  \mathrm{End}_{K^b(\Lambda G)} (\mathscr{C})$. Note that there is no canonical choice for $\widetilde v$, since it depends on the homotopy equivalence we choose between $\mathscr{C}$ and $D$. The image of $\widetilde v$ under this isomorphism has order $h_0$, which means that there exists a null-homotopic endomorphism $n$ of $\mathscr{C}$ such that $\widetilde{v}^{h_0}=1+n$. We can then argue as in \cite{Bon}: we consider the complex $\mathscr{D}$ obtained from $\mathscr{C}$ by removing any direct summand homotopy equivalent to zero. Since $\ell > h_0$, one can extract an $h_0$-th root of $1+n$ using the power series $\sqrt[h_0]{1+X}$. In other words, there exists a null-homotopic endomorphism $n'$ of $\mathscr{D}$ such that $(1+n')^{h_0} = 1+n$. This allows us to define the action of $v \in C_W(c\sigma)$ on $\mathscr{D}$ to be the endomorphism $\mathscr{V} = (1+n')^{-1} \widetilde{v}$. 

\sk

In summary, we have constructed a bounded complex $\mathscr{D}$ of finitely generated $(\Lambda G,\Lambda \T^{cF}\rtimes C_W(c\sigma))$-bimodules such that:

\begin{itemize}

\item $\mathscr{D}$ is a bounded complex whose terms are  projective as both left and right modules (since the order of $C_W(c\sigma)$ is invertible in $\Lambda$);

\item the restrictions of $\mathscr{D}$ and $C=b\Rgc(\Y,\Lambda)$ to the category of  $(\Lambda G, \Lambda \T^{cF})$-bimodules are homotopy equivalent;

\item under the isomorphism $\mathrm{End}_{K^b(k G)}(\overline{\mathscr{D}}) \simeq \mathrm{End}_{K^b(k G)}(\overline{C})$, the images of $v^{-1}$ and $F^\delta$ differ only by a nilpotent element.

\end{itemize}

\noindent In particular, we can identify the generalized $(\lambda)$-eigenspaces of $v^{-1}$ and $F^\delta$ on the cohomology groups of $\Y$. This means that the pair $(\mathscr{D},\mathscr{V})$ has the same properties as the pair $(\mathscr{C}, \mathscr{F})$ that are required in the proof of Theorem \ref{4thm1}. Consequently, for all $j= 0, \ldots, h_0 -1$, the generalized $(q^{-j\delta})$-eigenspace $\mathscr{D}_j$ of $\mathscr{V}$ on $\mathscr{D}$ is homotopy equivalent to the following complex of $\Lambda G$-modules:

\centers{$  0 \longrightarrow P _{m_\zeta} \longrightarrow{P}_{m_\zeta+1} \longrightarrow \cdots \longrightarrow {P}_{j-1} \longrightarrow {P}_j \longrightarrow 0$}

\noindent where  $\zeta$ stands for $\zeta_j$. Using this particular representative, we can now determine the planar embedding of the Brauer tree:

\begin{thm}\label{4thm2}Under the assumption $\mathrm{(S)}$, Conjecture \hyperref[2conj2]{\emph{(HLM+)}}  holds.
\end{thm}

\begin{proof} It is sufficient to show that the condition $\mathrm{Ext}_{kG}^1(S_{m_\zeta},S_{m_\xi}) \neq 0$ forces $m_\xi$ to be congruent to $M_\zeta +1$ modulo $h_0$. Using the formalism of derived categories, we shall write this group as

\centers{$ \mathrm{Ext}_{kG}^1(S_{m_\zeta},S_{m_\xi}) \, \simeq  \, \mathrm{Hom}_{D^b(kG)}(S_{m_\zeta},S_{m_\xi}[1])$}

\noindent which in turn is isomorphic to

\centers{$ \mathrm{Hom}_{D^b\big(k\T^{cF}\rtimes C_W(c\sigma)\big)}\big(\overline{\mathscr{D}}^\vee \otimes_{k G} S_{m_\zeta},\overline{\mathscr{D}}^\vee \otimes_{k G} S_{m_\xi}[1]\big)$}

\noindent via the fully-faithful functor $\overline{\mathscr{D}}^\vee \otimes_{k G} -$. Now, for $i \neq m_\zeta$, the module $\overline{P_i}^\vee \otimes_{kG} S_{m_\zeta} = \mathrm{Hom}_{kG}(\overline{P}_i,S_{m_\zeta})$ is zero. Consequently, we can use the action of $v$ to obtain the following isomorphisms in the category $K^b(kC_W(c\sigma)$-$\mathrm{mod})$: 

\centers{$ \overline{\mathscr{D}}^\vee \otimes_{k G} S_{m_\zeta} \, \simeq \, \displaystyle \bigoplus_{j=0}^{h_0-1} \overline{\mathscr{D}}_j^\vee \otimes_{k G} S_{m_\zeta} \,  \simeq \, \displaystyle \bigoplus_{j=m_\zeta}^{M_\zeta} \overline{\mathscr{D}}_j^\vee \otimes_{k G} S_{m_\zeta}  \, \simeq \, \bigoplus_{j=m_\zeta}^{M_\zeta} k_j[-r]$}

\noindent where $k_j$ is the one-dimentional simple $kC_W(c\sigma)$ module on which $v$ acts by multiplication by $\overline{q}^{-j\delta}$. Note that $v$ acts on $\T^{cF}$ by raising any element to the power of $q^{-\delta}$ so that this notation is consistent to the one given in Example \ref{2ex2}. From the previous isomorphism we deduce that the complex $ \overline{\mathscr{D}}^\vee \otimes_{k G} S_{m_\zeta}$ is quasi-isomorphic to a $k\T^{cF}\rtimes C_W(c\sigma)$-module $N_{m_\zeta}$ concentrated in degree $r$ satisfying

\centers{$ \mathrm{Res}_{kC_W(c\sigma)}^{k\T^{cF}\rtimes C_W(c\sigma)}\big(N_{m_\zeta}\big) \, \simeq \,  k_{m_\zeta} \oplus k_{m_\zeta+1} \oplus \cdots \oplus k_{M_\zeta}.$}

\noindent This shows that the composition factors of $N_{m_\zeta}$ are exactly $k_{m_\zeta}, \ldots, k_{M_\zeta}$. 

\sk

On the other side, the group $\mathrm{Ext}_{k\T^{cF} \rtimes C_W(c\sigma)}^1(k_i,k_j)  $ is non-zero if and only if  $j \equiv i+1$ modulo $h_0$ (see Example \ref{2ex2}). Since

\centers{$ \mathrm{Ext}_{kG}^1(S_{m_\zeta},S_{m_\xi}) \, \simeq  \, \mathrm{Ext}_{k\T^{cF} \rtimes C_W(c\sigma)}^1(N_{m_\zeta},N_{m_\xi})$}

\noindent we deduce that this latter group is non-zero only if there exist integers  $i \in \intn{m_\zeta}{M_\zeta}$ and $j \in \intn{m_\xi}{M_\xi}$ such that   $j \equiv i+1$ modulo $h_0$. We conclude by observing that the only integers that can satisfy this condition are $i = M_\zeta$ and $j=m_\xi$. 
\end{proof}

\begin{rmk} One can be more precise about the module $N_{m_\zeta}$. Being the image of an indecomposable, it is itself an indecomposable $k\T^{cF}\rtimes C_W(c\sigma)$-module. Moreover, since its composition factors are $k_{m_\zeta}, \ldots, k_{M_\zeta}$, it is necessarily uniserial and of the following shape:

\centers{$N_{m_\zeta} \, = \, \hskip-1.3mm\begin{array}{c} k_{M_\zeta} \\[3pt] k_{M_\zeta -1} \\[3pt] \vdots  \\[3pt] k_{m_\zeta} \\ \end{array} \hskip-1.3mm \ \ \cdot$}

\noindent Note however that there is still an  ambiguity for groups of type A$_1$.
\end{rmk}

\begin{rmk} For groups of $\mathbb{F}_q$-rank $1$ such  as the Ree group of type ${}^2 $G$_2$, the assumption $\mathrm{(S)}$ is automatically satisfied (the corresponding Deligne-Lusztig variety is an irreducible affine curve). As a consequence, the planar embedded Brauer tree of the principal $\ell$-block when $q$ has ordre $12$ modulo $\ell$ is exactly the one given in Figure \ref{fig3}. This completes \cite[Theorem 4.4]{Hiss}.

\sk

Furthermore, we shall prove in a subsequent paper that the assumption $\mathrm{(S)}$ holds whenever conjecture \hyperref[2conj1]{(HLM)} holds and $p$ is assumed to be a good prime number \cite{Du3}. In particular, from the knowledge of the shape of the Brauer tree we can deduce the planar embedding. We obtain therefore the planar embedding for groups of type F$_4$ when $p \neq 2,3$, which completes \cite[Theorem 2.1]{HL}. This would also apply to groups of type E$_7$ and E$_8$ if the shapes of the trees were known.

\end{rmk}

\subsection{Perverse equivalence and decomposition matrix\label{4se4}}

To finish off with the principal $\ell$-block, we observe that the equivalence induced by the cohomology of $\Y$ is perverse in the sense of \cite{CR1}. This leads to a conceptual proof of the unitriangularity shape of the decomposition matrix of the block.

\sk

Let us fix some notation for the objects involved in the definition. Recall that a full subcategory $\mathscr{B}$ of an abelian category $\mathscr{A}$ is a \emph{Serre subcategory} if for all short exact sequence

\centers{$0 \longrightarrow A \longrightarrow B \longrightarrow C \longrightarrow 0$} 

\noindent in $\mathscr{A}$, $B$ is an object of $\mathscr{B}$ if and only if both $A$ and $C$ are objects of $\mathscr{B}$. In other words, the category $\mathscr{B}$ is stable by quotients, submodules, and extensions by objects of $\mathscr{B}$.

\sk

Given such a category, we will denote by $D_\mathscr{B}^b(\a)$ the full subcategory of $D^b(\a)$ of complexes whose cohomology groups are objects of $\mathscr{B}$. We can then form the quotient categories $\a/\b$ (abelian) and  $D^b(\a)/D_{\mathscr{B}}^b(\a)$ (triangulated). 

\begin{de} Let $\a$ and $\a'$ be two abelian categories, $\mathscr{S}$ and $\mathscr{S}'$ the sets of isomorphism classes of simple objects of these categories. A derived equivalence $\Theta\, : \, D^b(\a) \mathop{\longrightarrow}\limits^\sim D^b(\a')$ is \emph{perverse} if there exist

\begin{itemize} 

\item filtrations $\emptyset = \mathscr{S}_0 \subset \mathscr{S}_1 \subset \cdots \subset \mathscr{S}_r = \mathscr{S}$ and  $\emptyset = \mathscr{S}_0' \subset \mathscr{S}_1' \subset \cdots \subset \mathscr{S}_r' = \mathscr{S}'$ of the sets $\mathscr{S}$ and $\mathscr{S}'$;

\item a perversity function $p : \intn{0}{r} \longrightarrow \mathbb{Z}$;

\end{itemize}

\noindent such that if $\a_i$ (resp. $\a_i'$) denotes the Serre subcategory generated by  $\mathscr{S}_i$ (resp. $\mathscr{S}_i'$), then for all $i$

\begin{enumerate} 

\item[$\mathrm{(i)}$] $\Theta$ restricts to an equivalence $D_{\a_i}^b(\a) \simeq D_{\a_i'}^b(\a')$;

\item[$\mathrm{(ii)}$] the functor \ $\a_i/\a_{i-1} \ \hookrightarrow \  D_{\a_i}^b(\a)/D_{\a_{i-1}}^b(\a) \mathop{\longrightarrow}\limits^\Theta D_{\a_i'}^b(\a')/D_{\a_{i-1}'}^b(\a')$ factors through the following commutative diagram:

\sk

\centers{$ \begin{psmatrix}%[colsep=13mm] 
D_{\a_i}^b(\a)/D_{\a_{i-1}}^b(\a) & D_{\a_i'}^b(\a')/D_{\a_{i-1}'}^b(\a') \\
\a_i /\a_{i-1} & \a_i' /\a_{i-1}'
\everypsbox{\scriptstyle}
\psset{arrows=->,nodesep=6pt}  
\ncline{1,1}{1,2}^{\Theta}
\psset{arrows=->,linestyle=dashed,nodesep=6pt} 
\ncline{2,1}{2,2}^{\sim}
\psset{arrows=H->,hooklength=1.5mm,hookwidth=-1.2mm,linestyle=solid,nodesep=6pt} 
\ncline{2,1}{1,1}<{\mathrm{can}}
\ncline{2,2}{1,2}>{[p(i)]}		
\end{psmatrix}$}

\noindent so that  $\Theta[-p(i)]$ induces an equivalence $\a_i / \a_{i-1} \simeq \a_i' / \a_{i-1}'$.
\end{enumerate}
\end{de}

From now on, we will assume that the all the objects of the categories involved have finite composition series. This allows us to reformulate the assertions $\mathrm{(i)}$ and $\mathrm{(ii)}$ into less abstract terms:

\begin{itemize}

\item for any simple object $L$ in $\mathscr{S}_i$, the composition factors of $\mathrm{H}^n(\Theta(L))$ lie in $\mathscr{S}_i'$, and even in $\mathscr{S}_{i-1}'$ whenever $n \neq -p(i)$;

\item if $L \in \mathscr{S}_i \smallsetminus \mathscr{S}_{i-1}$ then $\mathrm{H}^{-p(i)}(\Theta(L))$ has a unique composition factor $L'$ in $\mathscr{S}_{i}'\smallsetminus \mathscr{S}_{i-1}'$ and the map $L \longmapsto L'$ induces a bijection between the sets $\mathscr{S}_i\smallsetminus \mathscr{S}_{i-1}$ and $\mathscr{S}_i'\smallsetminus \mathscr{S}_{i-1}'$.

\end{itemize}

\noindent Roughly speaking, $\Theta(L)$ is quasi-isomorphic to $L'[p(i)]$ "modulo some composition factors in $\mathscr{S}_{i-1}'$".

\sk

Let us go back to the derived equivalence we have been studying. Recall that we have constructed a representative  $\mathscr{D}$ of the cohomology of $\Y$, as a bounded complex of finitely generated $(\Lambda G,\Lambda \T^{cF}\rtimes C_W(c\sigma))$-bimodules, both projective as left and right modules. We can argue as in the proof of Theorem \ref{4thm2} and show that for all simple $kG$-module in the principal block, say $S_j$, we have
\begin{equation} \overline{\mathscr{D}}^\vee \otimes_{k G} S_j \, \simeq \,  N_j[m_\zeta -j-r] \label{4eq1}\end{equation}
\noindent where $N_j$ is a $k\T^{cF}\rtimes C_W(c\sigma)$-module satisfying

\centers{$ \mathrm{Res}_{kC_W(c\sigma)}^{k\T^{cF}\rtimes C_W(c\sigma)}\big(N_j\big) \, \simeq \,  k_{j} \oplus k_{j+1} \oplus \cdots \oplus k_{M_{\zeta_j}}.$}

\noindent In particular, the composition factors of $N_j$ are exactly the (inflation of) the simple modules $k_j,\ldots,k_{M_{\zeta_j}}$.   

\sk

From that observation we shall deduce that the functor $\Theta : \overline{\mathscr{D}} \otimes_{kG} -$ induces a perverse equivalence. We use the height function associated to the Brauer tree to define the filtrations

\centers{$ \mathscr{S}_i \, = \, \big\{ S \in \mathrm{Irr}\, kG \, \big| \, \hg(S) \leq r-i\big\}$}

\leftcenters{and}{$\mathscr{S'}_i = \big\{ k_j \in \mathrm{Irr}\, k\T^{cF}\rtimes C_W(c\sigma) \, \big| \, S_j \in \mathscr{S}_i\big\}$}

\noindent on the set of isomorphism classes of simple modules (recall that $\mathrm{h}_\Gamma(S_j) = j-m_{\zeta_j})$. Note that $\mathscr{S}_r = \mathrm{Irr}\, kG$ and $\mathscr{S}_r' = \mathrm{Irr}\,kT^{cF}\rtimes C_W(c\sigma)$ are the last terms of these filtrations.

\sk

Finally, the image by $\Theta$ of a module $S_j \in \mathscr{S}_i \smallsetminus \mathscr{S}_{i-1}$ has by \ref{4eq1} a unique composition factor in $\mathscr{S}_i'\smallsetminus \mathscr{S}_{i-1}'$, namely $k_j$. Since it is concentrated in degree $r+j-m_\zeta = r+\hg(S_j) = 2r-i$, we may choose the perversity function to be $p(i) = i-2r$ so that following holds:

\begin{thm}\label{4thm4}Assume that the assumption $\mathrm{(S})$ holds. Then the functor $\Theta = \overline{\mathscr{D}} \otimes_{kG} -$ together with the triple $(\mathscr{S}_\bullet, \mathscr{S}_\bullet',p)$ induces a perverse equivalence between the principal blocks of  $kG$ and $k\T^{cF}\rtimes C_W(c\sigma)$. The corresponding bijection on the set of isomorphism classes of simple modules is $S_j \longleftrightarrow k_j$.
\end{thm}

In the case of the derived equivalences predicted by Brou\'e's conjecture, one of the groups involved can be written as a semi-direct product $D \rtimes E$ with $D$ an abelian $\ell$-group and $E$ an $\ell'$-subgroup of $\mathrm{Aut}(D)$. For such an equivalence, the existence of a perverse datum provides  information on the shape of the decomposition matrix of the block:

\begin{cor}[Chuang-Rouquier] There exists an ordering of the simple $KG$-modules and the simple $kG$-modules such that the decomposition matrix of the principal block of $\Lambda G$ has unitriangular shape.
\end{cor}

\begin{proof} We claim that any ordering compatible with the perversity (here, the height in  the tree) will be acceptable. Let $\widetilde \Theta = {\mathscr{D}} \otimes_{\Lambda G} -$. The functor 

\centers{$K \widetilde \Theta \, : \, D^b(KG$-$\mathrm{mod}) \longrightarrow D^b(K T^{cF}\rtimes C_W(c\sigma))$}  

\noindent induces a bijection on the simples lying in the principal blocks. It is clear that this bijection maps $\chi_j$ to $\eta_j$ (see Notation \ref{2not2} and Example \ref{2ex2}).

\sk

Let $\widetilde{\Theta}^*$ be an inverse functor of $\widetilde{\Theta}$. It induces again a perverse equivalence $\Theta^* = k\widetilde \Theta^*$ with respect to the same filtrations but for an opposite perversity function. Since the filtration is given by the height in the tree, we obtain

\centers{$ \dec_G \, \chi_j \, = \, \dec_G \, \widetilde{\Theta}^* (\eta_j) \, = \, \Theta^* \big([k_j]_k\big) \, = \,  [S_j]_k \, +  \hskip-3mm \displaystyle \sum_{\hg(S_i) \geq \hg(S_j)} \hskip -3mm a_{i,j} [S_i]_k $}

\noindent for some integers $a_{i,j}$, which proves that the decomposition matrix of $\Lambda G b$ has a unitriangular shape. \end{proof}

\begin{center} {\subsection*{Acknowledgements}} \end{center}

I would like to thank the Oxford Mathematical Institute for providing me excellent research environment during the early stage of this work. I wish to thank  Joe Chuang and Gerhard Hiss for our fruitful correspondence and C\'edric Bonnaf\'e, Jean-Fran\c{c}ois Dat and Rapha\"el Rouquier for many valuable comments and suggestions.

\bk

\bibliographystyle{abbrv}
\bibliography{coxorbits}

\end{document}